\documentclass[11pt,oneside]{amsart}
\usepackage[pagebackref,breaklinks,unicode]{hyperref} 
\title{CAT(0) square complexes that do not embed into finite products of trees}

\author{James Davies}
\address{Faculty of Mathematics and Computer Science, Leipzig University, Germany}
\email{jgdavies@uwaterloo.ca}

\author{Harry Petyt}
\address{Mathematics Institute, University of Warwick, UK}
\email{harrypetyt@gmail.com}

\usepackage{amsmath, amssymb, amsthm} 
\usepackage[english]{babel} 
\usepackage[font=small,justification=centering]{caption,subcaption} 
\usepackage[nodayofweek]{datetime}
\usepackage{enumitem} \setlist{nosep} 
\usepackage[T1]{fontenc} 
\usepackage[a4paper]{geometry} 
\usepackage[utf8]{inputenc} 
    \usepackage{csquotes} 
\usepackage{ifthen} 
\usepackage{mathabx} 
\usepackage{mathtools} 
\usepackage[dvipsnames]{xcolor} 
\usepackage{setspace} 
\usepackage{textcomp} 
\usepackage{tikz-cd} 
\usepackage{xfrac} 
\usepackage[capitalise,noabbrev]{cleveref}

\let\OLDthebibliography\thebibliography \renewcommand\thebibliography[1]{   
    \OLDthebibliography{#1}\setlength{\parskip}{0pt}\setlength{\itemsep}{0pt plus 0.3ex}} 

\makeatletter\def\subsection{\@startsection{subsection}{1}\z@{.7\linespacing\@plus\linespacing}
    {.5\linespacing}{\normalfont\scshape\centering}}\makeatother 

\renewcommand*{\backrefalt}[4]{\ifcase #1 (Not cited).\or (Cited p.~#2).\else (Cited pp.~#2).\fi} 

\newcounter{shcount}
\newcounter{thmcount}
\newcounter{enumlabelcount}
\newcounter{claimcount}

\newcommand*{\bsh}[1]{\theoremstyle{definition}\newtheorem{subhead\theshcount}[theorem]{#1}
    \begin{subhead\theshcount}} 
\newcommand*{\esh}{\end{subhead\theshcount}\stepcounter{shcount}} 
\newcommand*{\ubsh}[1]{\theoremstyle{definition}\newtheorem*{subhead\theshcount}{#1}
    \begin{subhead\theshcount}} 
\newcommand*{\uesh}{\end{subhead\theshcount}\stepcounter{shcount}} 

\newcommand*{\numberedtheorem}[3]{\theoremstyle{plain}\newtheorem*{makethm\thethmcount}{#1}
    \ifthenelse{\equal{#2}{}}{\begin{makethm\thethmcount}#3\end{makethm\thethmcount}\stepcounter{thmcount}}
    {\begin{makethm\thethmcount}[#2]#3\end{makethm\thethmcount}\stepcounter{thmcount}}} 

\makeatletter\newcommand\enumlabel[1][]{\item[#1]
    \refstepcounter{enumlabelcount}\def\@currentlabel{#1}}\makeatother

\newenvironment{claim*}{\medskip\noindent\textbf{Claim:}\hspace{0.5mm}}{}

\newenvironment{claim*proof}{\medskip\noindent\emph{Proof of Claim.}\hspace{0.5mm}}
    {\leavevmode\unskip\penalty9999\hbox{}\nobreak\hfill\quad\hbox{$\diamondsuit$}\medskip}

\newcommand*{\eps}{\varepsilon}

\newcommand*{\N}{\mathbb{N}}
\newcommand*{\R}{\mathbb{R}}

\newcommand*{\W}{\mathcal{W}}

\newcommand*{\cal}{\mathcal}

\newcommand*{\ssm}{\smallsetminus}

\newcommand{\Cross}{{\begin{tikzpicture}
    \draw (1ex,0) -- (1ex,1.8ex); \draw (0.4ex,1.2ex) -- (1.6ex,1.2ex);
    \end{tikzpicture}}}

\DeclareMathOperator{\dist}{\mathsf{d}}
\DeclareMathOperator{\hull}{\mathsf{Hull}}

\newcommand{\ignore}[2]{\left\{\kern-.7ex\left\{#1\right\}\kern-.7ex\right\}_{#2}}

\newcommand*{\mk}{\medskip}

\definecolor{harrycomment}{rgb}{0.6,0,0.4}

\definecolor{jamescomment}{rgb}{0,0.5,0.5}

\newtheorem{mthm}{Theorem} 
\newtheorem{mcor}[mthm]{Corollary} 
\newtheorem*{conjecture*}{Conjecture}

\newtheorem{question}{Question}
\swapnumbers
\newtheorem{theorem}{Theorem}[section]
\newtheorem{lemma}[theorem]{Lemma} 

\newtheorem{proposition}[theorem]{Proposition}
\theoremstyle{definition}
\newtheorem{definition}[theorem]{Definition}

\newtheorem*{remark*}{Remark}

\begin{document}

\begin{abstract}


Answering a question of Chepoi and Hagen, we give two constructions of bounded degree CAT(0) square complexes that cannot be isometrically embedded into any finite product of trees.

The first is based on Burling graphs and has degree at most five, which is optimal.
These square complexes topologically embed into $\R^3$ since we further prove that every CAT(0) square complex whose vertices have degree at most five can be topologically embedded into the product of a line and a star.
It turns out that in this setting, Burling graphs are the only obstruction:
we prove that CAT(0) square complexes that topologically embed into $\R^3$ and whose crossing graph forbids some induced Burling graph can be isometrically embedded into a finite product of trees.

Our second CAT(0) square complex has degree at most six and its crossing graph forbids an induced Burling graph.
Of independent graph-theoretic interest, along the way we construct intersections of chordal and interval graphs with arbitrarily large girth and chromatic number. We also discuss connections to nice labellings of event structures.
\end{abstract}

\maketitle

\section{Introduction}

CAT(0) cube complexes arise naturally in various settings as higher-dimensional analogues of trees. For example, from a metric perspective they can be characterised as the cube complexes whose piecewise-euclidean metric is \emph{CAT(0)}, the metric analogue of a simply connected manifold having nonpositive sectional curvature \cite{gromov:hyperbolic,leary:metric}. Equivalently they are the cube complexes whose piecewise-$\ell^\infty$ metric is \emph{injective} or \emph{hyperconvex}~\cite{miesch:injective}. In graph theory, the class of 1--skeletons of CAT(0) cube complexes is precisely the class of \emph{median graphs} \cite{chepoi:graphs}. Their 0--skeletons are precisely the algebraic objects \emph{discrete median algebras} \cite{roller:poc}. And in computer science they are (together with a base vertex) exactly the domains of \emph{event structures with binary conflict} \cite{barthelemyconstantin:median}.

We equip CAT(0) cube complexes with their piecewise-$\ell^1$ metric, which restricts to the graph metric on the 1--skeleton. It is reasonable to wonder how different these generalised trees are from trees. For instance, one may ask whether they can be isometrically embedded into a finite product of trees. It is not hard to see that this is not always the case. Given any simplicial graph $G$, the \emph{simplex graph} $k(G)$ is the graph with a vertex for each clique of $G$ and an edge whenever two cliques differ by a single vertex. The graph $k(G)$ is median, the dimension of the corresponding CAT(0) cube complex is equal to the clique number of $G$, and $k(G)$ isometrically embeds in a product of $k$ trees if and only if $G$ is $k$--colourable~\cite{bandeltvandevel:embedding}. The question can therefore be refuted by taking any sequence of graphs with bounded clique number and unbounded chromatic number.

The degree of the vertex of $k(G)$ corresponding to the empty clique is $|G|$, so increasing the chromatic number of $G$ also increases the maximal degree of vertices in $k(G)$. This leads to a natural refinement of the above question, which was independently raised by Chepoi, Haglund, Niblo, and Sageev (see \cite{chepoihagen:onembeddings}). When we say that a CAT(0) cube complex has \emph{bounded degree} (or \emph{degree at most $n$}), we mean that its vertices have bounded degree (or degree at most $n$).

\begin{question} \label{q:ie}
Let $Q$ be a CAT(0) cube complex of bounded degree. Does $Q$ isometrically embed into a finite product of trees?
\end{question}

Even this refined question turns out to have a negative answer. Indeed, Chepoi--Hagen gave an example of a 5--dimensional CAT(0) cube complex of degree at most 72 that cannot be isometrically embedded into any finite product of trees \cite{chepoihagen:onembeddings}. They asked whether a lower-dimensional example exists \cite[Qn~4]{chepoihagen:onembeddings}. Complementing this, they proved that if $Q$ is a \mbox{2--dimensional} CAT(0) cube complex (also known as a CAT(0) \emph{square} complex) of bounded degree and no \emph{vertex link} contains a 5--cycle, then $Q$ can be isometrically embedded into a finite product of trees \cite{chepoihagen:corrigendum}. We also remark that every planar CAT(0) cube complex isometrically embeds into a product of five trees \cite{bandeltchepoieppstein:combinatorics}, and all $\delta$--hyperbolic CAT(0) cube complexes of bounded degree isometrically embed into finite products of trees~\cite[Prop.~10.18]{haglund:aspects}.

Our first main result gives the strongest possible negative answer to Question~\ref{q:ie}. Recall that $K_{1,n}$ denotes the star-graph with $n$ leaves. We refer to Definition~\ref{def:outdeg} for the notions of pointing and out-degree.

\begin{mthm} \label{thm:bad_square_complex}
There is a CAT(0) square complex $Q$ of degree at most five that cannot be isometrically embedded into any finite product of trees.

Moreover, $Q$ can be topologically embedded into $K_{1,3}\times\R\subset\R^3$ and can be pointed to have maximum out-degree at most three.
\end{mthm}

The fact that the square complex in \cref{thm:bad_square_complex} topologically embeds into $\R^3$ is a special case of a more general phenomenon. 

\begin{mthm} \label{thm:5_in_3-space}
Every CAT(0) square complex $Q$ of degree at most five admits a topological embedding in $K_{1,4}\times\R$.

Furthermore, if $Q$ can be pointed to have maximum out-degree at most three, then $Q$ admits a topological embedding in $K_{1,3}\times\R$.
\end{mthm}

The degree at most five condition in Theorem~\ref{thm:5_in_3-space} is tight, as can be seen by considering the simplex graph $k(K_{3,3})$: it is a CAT(0) square complex of maximum degree six, and it has a non-planar vertex-link, so cannot be topologically embedded in $\R^3$. However, it may still be true that all CAT(0) cube complexes that can be pointed to have maximum out-degree at most three can be topologically embedded into $\R^3$.

The construction of the 5--dimensional CAT(0) cube complex in \cite{chepoihagen:onembeddings} is based on the family of \emph{Burling graphs}, which are a classical family of triangle-free graphs with large chromatic number \cite{burling:oncolouring}. More precisely, it starts from Burling's geometric representation of Burling graphs as intersection graphs of boxes in $\R^3$, ``lifts'' them to 4--dimensional boxes as in \cite{chepoi:nice}, then ``recubulates'' to obtain the final 5--dimensional example.
Our construction for \cref{thm:bad_square_complex} instead departs from a different geometric representation of Burling graphs as intersection graphs of particularly well behaved \emph{rectangular frames} in the plane \cite{pournajafitrotignon:burling:1}. We add additional frames (for the purposes of bounding degree), then apply \emph{Sageev's construction} (see Item~\ref{sh:sageev}).

The relevance of Burling graphs to embedding a CAT(0) cube complex $Q$ is via the \emph{crossing graph} of $Q$ (Definition~\ref{def:crossing}). The dimension of $Q$ is equal to the clique number of its crossing graph, and $Q$ can be isometrically embedded into a product of $k$ trees if and only if its crossing graph can be $k$--coloured (Lemma~\ref{lem:crossing_colour}).

Despite Burling graphs being a classical family, it has only recently become clear how fundamental they are. They are the only known hereditary class of graphs that is minimal with respect to having unbounded chromatic number, other than complete graphs~\cite{abrishamibrianskidaviesdumasarikovarzazewskiwalczak:burling}. This means that just as how having large cliques is a ``weak'' obstruction to having bounded chromatic number, so is containing Burling graphs. Indeed, recent work has shown that in various graph classes, cliques and Burling graphs are the only two obstructions to having bounded chromatic number~\cite{abrishamibrianskidaviesdumasarikovarzazewskiwalczak:burling}. Such classes are called \emph{Burling-controlled}.

Our next result establishes Burling-control for CAT(0) cube complexes that can be topologically embedded into $\R^3$.  

\begin{mthm}\label{thm:cubeR3}
For each Burling graph $B$ there is an integer $k$ such that the following holds. If $Q$ is a CAT(0) cube complex that topologically embeds into $\mathbb{R}^3$ and whose crossing graph excludes $B$ as an induced subgraph, then $Q$ isometrically embeds into a product of $k$ trees.
\end{mthm}

Combining Theorems~\ref{thm:5_in_3-space} and~\ref{thm:cubeR3} gives the following corollary whose assumptions are more intrinsic to the cube complex.

\begin{mcor} \label{thm:5}
For each Burling graph $B$ there is an integer $k$ such that the following holds. If $Q$ is a CAT(0) square complex with degree at most five and whose crossing graph excludes $B$ as an induced subgraph, then $Q$ isometrically embeds into a product of $k$ trees.
\end{mcor}

We believe that it is possible to drop the 2--dimensionality assumption in \cref{thm:5}, by decomposing $Q$ along separating faces into parts that are either cubes or $3$--dimensional CAT(0) cube complexes that admit topological embeddings into $\mathbb{R}^3$, but we wish to avoid the additional technical details.

These results are suggestive that perhaps Burling graphs are the only reason that Question~\ref{q:ie} has a negative answer.
Using an entirely different construction to before, we show that this is not the case: there are other sources of failure for Question~\ref{q:ie}. The following is a weakened version of \cref{thm:6}.

\begin{mthm} \label{thm:bad_square_complex2}
There is a CAT(0) square complex $Q$ of degree at most six such that the crossing graph of $Q$ excludes some Burling graph as an induced subgraph but $Q$ cannot be isometrically embedded into any finite product of trees.
\end{mthm}

Note that degree at most six is optimal, by \cref{thm:5}. Moreover, by \cref{thm:cubeR3} the CAT(0) square complex $Q$ in \cref{thm:bad_square_complex2} does not topologically embed into $\R^3$. 
Of course, there are likely more direct ways to show that $Q$ does not topologically embed into $\R^3$, for instance by considering the forbidden minors for topologically embedding simply connected 2--dimensional simplicial complexes into $\R^3$ \cite{carmesin:embedding}.

Along the way to proving Theorem~\ref{thm:bad_square_complex2} we also happen to prove a result of independent graph-theoretic interest. 
An \emph{interval graph} is the intersection graph of a collection of intervals in $\R$. More generally, a \emph{chordal graph} is the intersection graph of a collection of subtrees of a tree. The chromatic number of a chordal graph is equal to its clique number.
The \emph{intersection} of two graphs with a common vertex set is the graph with the same vertex set and edge set being the common edges of both graphs.
Gy{\'a}rf{\'a}s \cite{gyarfas:problems} asked whether cliques are the only obstruction to intersections of chordal and interval graphs having bounded chromatic number.
This was answered in the negative by the results of~\cite{felsnerjoretmicektrotterwiechert:burling} (as observed in \cite{chaniotismiraftabspirkl:graphs}) by using Burling graphs (see also \cite{chaniotiskoertsspirkl:intersections}).
We show that intersection graphs of chordal and interval graphs are not even Burling-controlled.
More strongly, we show that such graphs can have arbitrarily large girth and chromatic number (the \emph{girth} of a graph is the length of its shortest cycle). See Theorem~\ref{thm:treerestricted}.

\begin{mthm}\label{thm:chordalinterval}
There are intersections of chordal and interval graphs that have arbitrarily large girth and chromatic number.
\end{mthm}

Our construction for Theorem~\ref{thm:chordalinterval} is a modification of Tutte's construction of triangle-free graphs of large chromatic number \cite{descartes:three,descartes:solution}, making use of a sparse version of the Hales--Jewett theorem \cite{halesjewett:regularity} due to Prömel and Voigt \cite{promelvoigt:sparse:grahamrothschild}.
Such modifications of Tutte's construction using (sparse versions \cite{promelvoigt:sparse:grahamrothschild,promelvoigt:sparse:galaiwitt,rodl:onramsey} of) Ramsey results such as Gallai's theorem \cite{rado:note} and the Hales--Jewett theorem \cite{halesjewett:regularity} were introduced by the first author \cite{davies:box} and have since found further uses in \cite{davieskellerkleistsmorodinskywalczak:solution,davieshatzelyepremyan:counterexample,davieshatzelyepremyan:minimal}.

Finally, let us quickly mention a connection to theoretical computer science. In \cite{chepoi:nice}, Chepoi disproved the \emph{nice labelling conjecture} of Rozoy and Thiagarajan \cite{rozoythiagarajan:event}, which stated that ``every event structure with finite degree admits a nice labelling with a finite number of labels''. His construction has degree five and uses the standard correspondence between event structures and CAT(0) cube complexes \cite{barthelemyconstantin:median,chepoi:graphs,roller:poc}. By contrast, it is known that the conjecture holds for event structures of degree two \cite{assousbouchittecharrettonrozoy:finite} and in various cases of degree three \cite{santocanale:nice}.

The existence of a pointed CAT(0) cube complex with maximum out-degree at most $n$ that cannot be isometrically embedded into any finite product of trees in particular implies the existence of an event structure with degree at most $n$ that admits no nice labelling with finitely many labels (see \cite[\S3.5]{chepoihagen:onembeddings} for further details). The following is therefore a corollary of \cref{thm:bad_square_complex}.

\begin{mcor} \label{thm:nice}
There is an event structure of degree three that has no nice labelling with finitely many labels.
\end{mcor}

A simple median argument that can be found in \cite[\S7.3]{chepoihagen:onembeddings} shows that if $Q$ is a $n$--dimensional CAT(0) cube complex that can be pointed to have maximal out-degree $d$, then $Q$ has maximal degree at most $d+n$. It therefore follows from Corollary~\ref{thm:5} that Burling graphs are the only obstruction to nice labelling for event structures of degree three whose associated CAT(0) cube complexes are 2--dimensional. It is possible that this holds more generally, which would be implied by the following.

\begin{conjecture*} 
For each Burling graph $B$ there is an integer $k$ such that the following holds. If a CAT(0) cube complex $Q$ whose crossing graph excludes $B$ as an induced subgraph can be pointed to have maximum out-degree three, then $Q$ isometrically embeds into a product of $k$ trees.
\end{conjecture*}

We remark that  the CAT(0) square complexes constructed in \cref{thm:bad_square_complex2} can be pointed so that they have maximum out-degree at most four, so the degree bound in the above conjecture cannot be loosened. 

In view of \cite{chepoihagen:onembeddings} and Theorems~\ref{thm:bad_square_complex} and~\ref{thm:bad_square_complex2}, it seems natural to ask the following, in which isometric embeddings are replaced by the more flexible quasiisometric embeddings (a map is a \emph{quasiisometric embedding} if it preserves distances up to affine error).

\begin{question} \label{q:qie}
Is there a CAT(0) cube complex of bounded degree that cannot be quasiisometrically embedded into any finite product of trees?
\end{question}

We remark that a straightforward modification of the construction of Theorem~\ref{thm:bad_square_complex} with barycentric subdivisions should give an example of a CAT(0) square complex of maximum degree five that does not admit any quasimedian quasiisometric embedding into any finite product of trees (a map is \emph{quasimedian} if it preserves medians up to bounded distance). There are also $n$--dimensional CAT(0) cube complexes of bounded degree (certain \emph{right-angled Artin groups}) that cannot be quasiisometrically embedded into any product of $n$ trees \cite{baderbensaidpetyt:quasiisometric:rigidity}. But we do not know of an example of a finite-dimensional CAT(0) cube complex, even of unbounded degree, that cannot be quasiisometrically embedded into any finite product of trees.

\medskip

\cref{sec:pre} contains preliminaries on CAT(0) cube complexes. After this, the remaining sections can essentially be read independently of one another.
In \cref{sec:frames} we use rectangular frames to prove Theorem~\ref{thm:burbox}, which is the main part of \cref{thm:bad_square_complex}. A more explicit description of the collections of frames involved is given in Appendix~\ref{sec:frames:1}. \cref{sec:R3} is dedicated to proving \cref{thm:5_in_3-space}, which together with Theorem~\ref{thm:burbox} establishes \cref{thm:bad_square_complex}. The short \cref{sec:string} proves \cref{thm:cubeR3}, using a result from \cite{abrishamibrianskidaviesdumasarikovarzazewskiwalczak:burling} that \emph{string graphs} are Burling-controlled.
Finally, Theorems~\ref{thm:bad_square_complex2} and~\ref{thm:chordalinterval} are proved in \cref{sec:noBur}.






\subsection*{Acknowledgements}

We are grateful to Victor Chepoi and Mark Hagen for helpful conversations about their work \cite{chepoihagen:onembeddings}. We thank the Isaac Newton Institute for their hospitality during the programme \emph{Operators, Graphs, Groups}, where this project started (EPSRC grant EP/Z000580/1). 
Research of the first author was supported by the Alexander von Humboldt Foundation in the framework of the Alexander von Humboldt Professorship of Daniel Král' endowed by the Federal Ministry of Education and Research.

There was no use of AI in the production of this paper.

\section{CAT(0) cube complexes}\label{sec:pre}

We refer the reader to the books \cite{wise:structure,bowditch:median:book,genevois:algebraic} for background on CAT(0) cube complexes. We provide some brief basics that are relevant to us.

A \emph{cube} is a copy of $[0,1]^n$ for some $n$. A \emph{cube complex} is a space built by isometrically gluing cubes along faces. The \emph{link} of a vertex $v$ in a cube complex $X$ is a cell complex that can be thought of as the ``$\eps$--sphere'' around $v$: it has a vertex for each end of an edge incident to $v$, and a collection of $n+1$ vertices in the link span an $n$--simplex if and only if there is an $(n+1)$--cube containing those edges. The cube complex is \emph{CAT(0)} if every vertex link is a flag simplicial complex.

Equivalently, a cube complex is \emph{CAT(0)} if it is contractible and its 1--skeleton is a \emph{median graph}: for any three vertices $x_1,x_2,x_3$, there is a unique vertex $\mu=\mu(x_1,x_2,x_3)$ such that $\dist(x_i,x_j)=\dist(x_i,m)+\dist(m,x_j)$ whenever $i\ne j$.

A \emph{midcube} in a CAT(0) cube complex is a subspace of a cube obtained by restricting one coordinate to the value $\frac12$. A \emph{hyperplane} is a connected, nonempty subset $h$ that intersects each cube in either the empty set or a midcube. If $e=xy$ is an edge of $Q$, then for each $z\in Q$ we have $\mu(x,y,z)\in\{x,y\}$. If $e$ intersects $h$, then an edge $x'y'$ intersects $h$ if and only if $\mu(x,y,x')\ne\mu(x,y,y')$.

\begin{definition}[Crossing graph] \label{def:crossing}
Let $Q$ be a CAT(0) cube complex. The \emph{crossing graph} $\Cross(Q)$ of $Q$ is the graph with a vertex for each hyperplane of $Q$ and an edge joining two vertices whenever the corresponding hyperplanes intersect. 
\end{definition}

\bsh{Sageev's construction} \label{sh:sageev}
Let $S$ be a set and let $P$ be a collection of bipartitions of $S$, called \emph{walls}. That is, each $h\in P$ is a partition $S=h^-\sqcup h^+$ into two \emph{halfspaces}. We call the pair $(S,P)$ a \emph{wallspace} if for each $s,t\in S$ there are only finitely many walls $h\in P$ for which $s$ and $t$ do not lie in the same halfspace.

An \emph{ultrafilter} $\phi$ is a choice of halfspace $\phi(h)$ for each $h\in P$ such that $\phi(h_1)\cap\phi(h_2)\ne\varnothing$ for all $h_1,h_2\in P$. We can define an extended metric on the set of ultrafilters by setting
\[
\dist(\phi,\psi) \,=\, |\{h\in P\,:\, \phi(h)\ne\psi(h)\}|.
\]
Each $s\in S$ determines a \emph{principal} ultrafilter $\phi_s$, where $\phi_s(h)$ is the halfspace containing $s$. 

The CAT(0) cube complex \emph{dual} to a wallspace $(S,P)$ has vertex set
\[
Q \,=\, \{\text{ultrafilters }\phi\,:\,\dist(\phi,\phi_s)<\infty \text{ for some }s\in S\},
\]
and an edge between vertices whenever they are at distance 1.
\esh

Let $(S,P)$ be a wallspace. We say that two walls $h_1,h_2\in P$ \emph{cross} if all four halfspaces $h_1^\pm\cap h_2^\pm$ are nonempty. The dimension of the CAT(0) cube complex dual to $(S,P)$ is equal to the supremal cardinality of a collection of pairwise crossing walls. 

Given a vertex $\phi$ in the CAT(0) cube complex $Q$ dual to $(S,P)$, we can define a partial order $<_\phi$ on $P$ by declaring $h_1<_\phi h_2$ whenever $\phi(h_1)\subset\phi(h_2)$.

\begin{lemma} \label{lem:degree}
The degree of the vertex $\phi$ in the CAT(0) cube complex $Q$ is equal to the number of minimal elements of $(P,<_\phi)$.
\end{lemma}

\begin{proof}
If $\psi$ is adjacent to $\phi$, then there is a unique wall $h$ such that $\psi(h)\ne\phi(h)$. If there were $h'\in P$ with $h'<_\phi h$, then we would have $\psi(h')\cap\psi(h)=\varnothing$, which is impossible because $\psi$ is an ultrafilter. Hence $h$ is a minimal element of $(P,<_\phi)$. 

Conversely, given a minimal element $h$ of $(P,<_\phi)$, let $\psi$ be defined by $\psi(h')=\phi(h')$ for all $h'\in P\ssm\{h\}$ and $\psi(h)\ne\phi(h)$. It is straightforward to check that $\psi$ is an ultrafilter, and it is adjacent to $\phi$ in $Q$.
\end{proof}

The following well-known lemma is a straightforward consequence of Sageev's construction; the below statement first appears in \cite[Lem.~1.48]{holloway:embeddings}, but similar statements can be found in \cite{dranishnikovjanuszkiewicz:every,bandeltchepoieppstein:ramified,button:groups,baderbensaidpetyt:quasiisometric:rigidity,genevois:median}. A more general statement about \emph{metric median algebras} is given in \cite{bowditch:embedding}. 

\begin{lemma} \label{lem:crossing_colour}
Let $Q$ be a CAT(0) cube complex. The crossing graph $\Cross(Q)$ is $k$--colourable if and only if $Q$ can be isometrically embedded into a product of $k$ trees.
\end{lemma}

\begin{definition}[Out-degree] \label{def:outdeg}
A \emph{pointed} CAT(0) cube complex is a CAT(0) cube complex $Q$ together with a chosen vertex $v_0\in Q$. Since median graphs are bipartite, if $xy$ is an edge of $Q$, then $x$ and $y$ are not equidistant from $v_0$. For a vertex $v\in Q$, the \emph{in-degree} of $v$ is the number of its neighbours $v'$ with $\dist(v_0,v')<\dist(v_0,v)$. The \emph{out-degree} of $v$ is the number of its neighbours $v'$ with $\dist(v_0,v')>\dist(v_0,v)$.
\end{definition}

\section{Frames} \label{sec:frames}

In this section we discuss graphs appearing as intersection graphs of collections of frames, and use them to construct the CAT(0) square complex appearing in \cref{thm:bad_square_complex}.

\begin{definition}[Frame]
A \emph{frame} $R$ is a subset of $\R^2$ that is the boundary of a region of the form $[a,b]\times[c,d]$ or $[a,\infty)\times[c,d]$, where $a<b$ and $c<d$. We say that $R$ is \emph{bounded} in the former case, and \emph{unbounded} in the latter case. The \emph{outside} of a frame is the component $R^-$ of $\R^2\ssm R$ containing $(a-1,c-1)$, and the \emph{inside} is $R^+=\R^2\ssm R^-$.
\end{definition}

Note that formally the frame $R$ is contained in its inside. Each frame defines a bipartition of the plane as $\R^2=R^+\sqcup R^-$. When we say that two frames $R_1$ and $R_2$ intersect, we exactly mean that $R_1\cap R_2\ne\varnothing$. We say that a frame $R_1$ is \emph{nested} in a frame $R_2$ if $R_1$ is contained in the interior of $R_2^+$. Note that nested frames do not intersect. More generally, we say that $R_1$ \emph{begins} in $R_2$ if the left edge of $R_1$ is contained in the interior of $R_2^+$.


\begin{definition}[Fair, strict]
We say a collection $\cal R$ of frames is \emph{fair} if whenever $R_1\in\cal R$ and $R_2\in\cal R$ intersect, either the left edge of $R_1$ is contained in the interior of $R_2^+$ and the right edge of $R_1$ (if it exists) is contained in $R_2^-$, or vice versa.

The collection $\cal R$ is \emph{triangle-free} if no three frames intersect pairwise, and \emph{connected} if $\bigcup_{R\in\cal R}R$ is connected.
See Figure~\ref{fig:strict}, which also appears as \cite[Fig.~1]{rzazewskiwalczak:polynomial}. 

A collection of frames is \emph{strict} if it is fair, connected, and for any two intersecting frames $R_1$ and $R_2$, there is no other frame whose left edge lies in $R_1^+\cap R_2^+$. 
\end{definition}

\begin{figure}[ht]
    \centering
    \includegraphics[height=20mm]{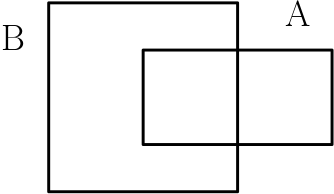}\quad
    \includegraphics[height=20mm]{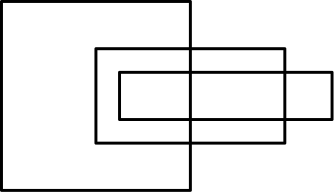}\quad
    \includegraphics[height=20mm]{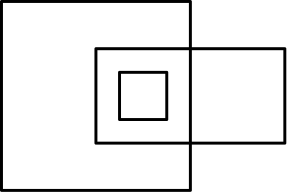}\quad
    \includegraphics[height=20mm]{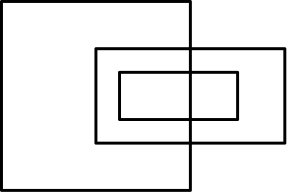}\quad
    \caption{(a) Required configuration of every intersecting pair of frames in a fair collection. (b) Disallowed in a triangle-free collection. (b)--(d) Disallowed configurations in a strict collection.} \label{fig:strict}
\end{figure}

If two frames in a fair collection intersect, then they must intersect as in Figure~\ref{fig:strict}. For frames $A$ and $B$ as on the left of Figure~\ref{fig:strict}; we say that frame $A$ \emph{exits} frame $B$. Note that a strict collection of frames is automatically triangle-free.
A collection $\cal R$ of frames is connected if and only if its intersection graph is connected. 

The relevance of strict collections of frames for us will come from their connection with Burling graphs. Indeed, it follows from the construction of \cite{pawlikkozikkrawczyklasonmicektrotterwalczak:triangle:geometric} that every Burling graph is the intersection graph of a fair collection of frames, and it was shown in \cite{pournajafitrotignon:burling:1} that Burling graphs are exactly the intersection graphs of strict collections of frames. 


Given a frame $R$, let $I(R)=[c,d]$ be the compact interval obtained by projecting $R$ to its second coordinate. If $\cal R$ is a collection of frames, then write $\cal I=\{I(R)\,:\,R\in\cal R\}$. The following is a direct consequence of the definition of fairness.

\begin{lemma} \label{lem:stern}
If $\cal R$ is a fair, connected collection of frames, then for any two intersecting elements of $\cal I$, one is contained in the interior of the other. Moreover, if $I(R_1)\subset I(R_2)$, then the left edge of $R_1$ is to the right of the left edge of $R_2$.
\end{lemma}

In particular, if $\cal R$ is fair and connected, then $\cal I$ has a unique maximal element.

Let $\cal R$ be a fair collection of frames. Given $R\in\cal R$, let $\cal I_R\subset\cal I$ be the sub-poset consisting of all elements properly contained in $I(R)$. We say that $R_1\in\cal R$ is \emph{outermost in $R_2\in\cal R$} if $I(R_1)$ is a maximal element of $\cal I_{R_2}$. 


\begin{lemma}\label{lem:firm}
For every finite, strict collection $\cal R$ of frames, there exists a finite collection $\cal D=\cal D_1\cup\cal D_2$ of unbounded frames such that:
\begin{itemize}
\item   $\cal R\cup\cal D$ is fair, connected, and triangle-free;
\item   each element of $\cal R$ has at most one outermost frame in $\cal R\cup\cal D$, which is in $\cal D$;
\item   each element of $\cal D_1$ has exactly one outermost frame in $\cal R\cup\cal D$, which is in $\cal R$; 
\item   each element of $\cal D_2$ has at most two outermost frames in $\cal R\cup\cal D$, both from $\cal D$; 
\end{itemize} 
\end{lemma}

\begin{proof}
For each $R\in \mathcal{R}$, there exists some maximal collection $R_1,\ldots ,R_{n(R)}\in \mathcal{R}$ such that $I(R_1),\ldots , I(R_{n(R)})\subset I(R)$ are maximal. 
Note that $I(R_1),\ldots , I(R_{n(R)})$ are disjoint. 
We may assume without loss of generality that $I(R_1)< \cdots < I(R_{n(R)})$.
Moreover, by strictness of $\cal R$, no $R_i$ is nested in any frame that $R$ exits (see Figure~\ref{fig:strict}).

We can therefore choose a collection of pairwise non-intersecting unbounded frames $\mathcal{D}_R=\{D_{R,1}, D_{R,1}', \ldots , D_{R,{n(R)}}, D_{R,{n(R)}}'\}$ with left edges contained in the interior of $R^+$ but not the inside of any frame that $R$ exits, 
and such that: 
\begin{itemize}
\item   $R_i$ nests in $D_{R,j}$ if and only if $i=j$; and
\item   $D_{R,i}$ nests in $D_{R,j}'$ if and only if $i\le j$.
\end{itemize} 
See Figure~\ref{fig:firm} for an illustration. Let $\cal D_{R,1}=\{D_{R,1},\dots,D_{R,n(R)}\}$ and $\cal D_{R,2}=\{D'_{R,1},\dots,D'_{R,n(R)}\}$. For $i\in\{1,2\}$, let $\mathcal{D}_i=\bigcup_{R\in\cal R}\cal D_{R,i}$. Set $\cal D=\cal D_1\cup\cal D_2$.

It is clear from the construction that $\cal R\cup\cal D$ is fair and connected. Also, if $\cal R$ is triangle-free, then for each $R\in\cal R$ there can be no $R_i$ with $i\in[n(R)]$ such that $R_i$ 
no element of $\cal D$ can intersect two intersecting elements of $\cal R$, so if $\cal R$ is triangle-free then so is $\cal R\cup\cal D$. 

For every $R\in \cal R$, the frame $D_{R,n(R)}'$ is the unique outermost frame of $\cal R \cup \cal D$ in $R$. 
For every $D_{R,i} \in \cal D$, the frame $R_i$ is the unique outermost frame of $\cal R \cup \cal D$ in $D_{R,i}$. 
Finally, for every $D_{R,i}'\in \cal D$, the frames $D_{R,i}$ and $D'_{R,i-1}$ (when $i>1$) are the only outermost frames of $\cal R \cup \cal D$ in $D_{R,i}'$. This completes the proof.
\end{proof}

\begin{figure}[ht]
\centering\includegraphics[height=8cm]{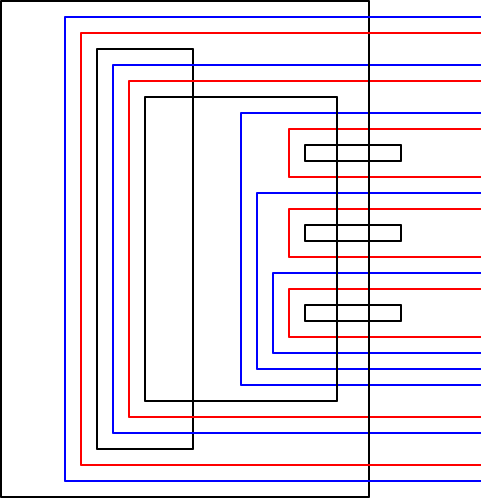}
\caption{Frames $\mathcal{R} \cup \mathcal{D}$ as in the proof of Lemma  \ref{lem:firm}. Frames in $\mathcal{R}$ are black. Frames $D_{R,i}\in\cal D$ are red, and frames $D'_{R,i}\in\mathcal{D}$ are blue.
} \label{fig:firm}
\end{figure}

\bsh{Construction} \label{sh:construction}
Let $\cal R$ be a finite, fair collection of frames. By viewing each $R\in\cal R$ as a bipartition of the plane, $\R^2=R^+\sqcup R^-$, we obtain a wallspace $(\R^2,\cal R)$. As described in Item~\ref{sh:sageev}, this wallspace has a dual CAT(0) cube complex $Q$. Two walls in $Q$ cross if and only if the corresponding frames in $\cal R$ intersect, so the crossing graph of $Q$ is the intersection graph of $\cal R$. In particular, if $\cal R$ is triangle-free, then $\dim Q\le2$.
\esh

\begin{lemma}\label{lem:burdeg}
Let $\cal R$ be a finite, strict collection of frames, and let $\cal D$ be as in \cref{lem:firm}. Let $Q$ be the CAT(0) square complex dual to $(\R^2,\cal R\cup\cal D)$, as described in Item~\ref{sh:construction}. The degree of each vertex of $Q$ is at most five.

Moreover, $Q$ has a vertex $o$ of degree two such that the directed graph $Q_o$ has maximum out-degree at most three.
\end{lemma}

\begin{proof}
Let us write $\cal R'=\cal R\cup\cal D$. By Lemma~\ref{lem:degree}, given $\phi\in Q$ we must bound the number of minimal elements of $(\cal R',<_\phi)$. We proceed in two steps. First we show that the map $\R^2\to Q$, given by sending $p\in\R^2$ to the principal ultrafilter $\phi_p$ it defines, is onto. We then use the properties of $\cal R'$ to understand the minimal elements of $(\cal R',<_{\phi_p})$ for $p\in\R^2$.

Let $\phi\in Q$, and let $p\in\R^2$ minimise $d=\dist(\phi,\phi_p)$. If $d=0$ then $\phi=\phi_p$. Otherwise, let $P=\{R\in\cal R'\,:\,\phi(R)\ne\phi_p(R)\}$. It is finite and nonempty. Let $R$ be one of the minimal elements of $(P,<_{\phi_p})$. Let us show that there is a region of $\R^2$ that defines an orientation of $\cal R'$ that differs from $\phi_p$ only on $R$. 

If there is not such region, then there is some minimal collection $R_1,\dots,R_k$ of elements of $\cal R'$ separating $p$ from $R$. These frames separate $p$ from the top side of $R$. If $p$ is inside $R$, then, after relabelling the $R_i$, fairness of $\cal R'$ implies that $p$ is inside $R_1$ and either: $k=1$ and $R_1$ is nested in $R$; or $R_1$ exits $R$. In the latter case, the fact that $\cal R'$ is triangle-free implies that $R_1,\dots,R_k$ cannot separate $p$ from the right edge of $R$, a contradiction. Thus $k=1$ in this case. A similar analysis when $p$ is outside $R$ again implies that $k=1$. But by minimality of $R$, we must now have that $\phi(R_1)=\phi_p(R_1)$, which contradicts the fact that $\phi$ is an ultrafilter. 

We have shown that there is a region of $\R^2$ that defines an orientation of $\cal R'$ that differs from $\phi_p$ only on $R$. If $q$ is a point in it, then by construction we have $\dist(\phi_q,\phi)=d-1$, contradicting the choice of $p$. Thus each vertex of $Q$ is represented by at least one region of $\R^2$. 

\medskip 

Controlling the number of minimal elements of $(\cal R',<_\phi)$ is equivalent to controlling the number of $R\in\cal R'$ that bound regions of $\R^2$ representing $\phi$. Firstly, consider a point $p$ on the outside of all frames. By \cref{lem:stern}, $\cal I$ has a unique maximal element, corresponding to some $R\in\cal R$. By \cref{lem:firm}, there is a unique element of $\cal R'$ that is outermost in $R$, and it is in $\cal D$, hence unbounded. Thus the degree of $\phi_p\in Q$ is two. 

Given any other $p\in\R^2$, we shall describe a set of at most five frames that bound the region in which $p$ lies. See Figure~\ref{fig:degree}. 

\begin{figure}[ht]
\begin{center} \makebox[0pt]{\begin{minipage}{1.3\textwidth}
\centering \includegraphics[height=50mm, trim = 0 2mm 0 3mm]{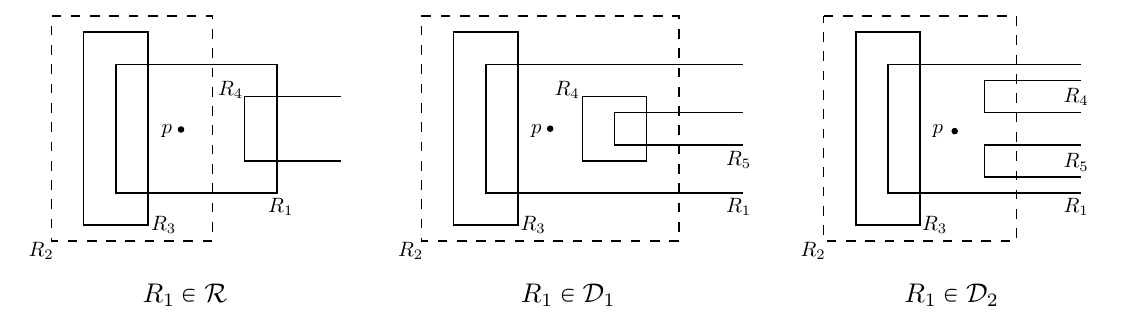}
\end{minipage}}\end{center}
\caption{The possible frames bounding the region of $\R^2$ containing a point $p$ in the proof of \cref{lem:burdeg}. (In the case $R_1\in\cal D_1$, the frame $R_4$ could also exit~$R_3$.) Note that although these frames account for all edges of $Q$ containing $\phi_p$, they do not account for all squares of $Q$ containing $\phi_p$. 
}\label{fig:degree}
\end{figure}

Let $R_1\in\cal R'$ be the frame whose left edge is rightmost among frames that $p$ is inside. Let $R_2\in\cal R$ be the frame whose right edge is leftmost among frames that $p$ is inside, if it exists. It can happen that $R_2=R_1$. By construction, $R_1$ and $R_2$ are the minimal elements of $\cal R'$ that $p$ is inside. We therefore need to understand which elements of $\cal R'$ that $p$ is outside can meet $R_1^+\cap R_2^+$. 

By fairness, any frame meeting the left, top, or bottom edge of $R_1$ has the left edge of $R_1$ in its inside. That is, $R_1$ begins in it. Let $R_3$ be the frame whose right edge is rightmost among frames with $p$ on the outside that $R_1$ begins in, if it exists. If $R_2\ne R_1$, then $R_3$ is nested in $R_2$, for otherwise $R_1,R_2,R_3$ would intersect pairwise.

The only other frames that we need to consider are those that begin in $R_1$ and $R_2$ (if it exists), which may either be nested in $R_1$ or exit it. Note that by minimality of $R_1$, the point $p$ lies outside every frame that begins in $R_1$.  

If $R_1\in\cal R$, then it has at most one outermost frame $R_4\in\cal R'$, which is in $\cal D$. Since $\cal R\cup\cal D$ is triangle-free, it is disjoint from $R_2$ if the latter exists. All elements of $\cal D$ are unbounded, so every other element of $\cal R'$ that begins in $R_1$ is nested in $R_4$.

If $R_1\in\cal D_1$, then it has exactly one outermost frame $R_4$ in $\cal R'$, which is in $\cal R$. If $R_4$ is unbounded, then every other element of $\cal R'$ that begins in $R_1$ is nested in $R_4$. Otherwise, $R_4$ has at most one outermost frame $R_5\in\cal R'$, which is in $\cal D$, hence unbounded. Every element of $\cal R'$ that begins in $R_1$ other than $R_4$ and $R_5$ is nested in $R_5$.

Finally, if $R_1\in\cal D_2$, then it has exactly two outermost frames $R_4$ and $R_5$ in $\cal R'$, both of which are in $\cal D$, hence unbounded. Every element of $\cal R'$ that begins in $R_1$ other than $R_4$ and $R_5$ is nested in either $R_4$ or $R_5$.

In all cases, we have $p\in R_1^+\cap R_2^+\cap R_3^-\cap R_4^-\cap R_5^-$ by construction. The only way that $p$ affects the definition of $R_1,\dots,R_5$ is in terms of the data of whether it is inside or outside each element of $\cal R$, or in other words in terms of the principal ultrafilter $\phi_p$. Hence if $p'$ is another point with $\phi_{p'}=\phi_p$, then the list $R_1,\dots,R_5$ obtained from $p'$ will be the same. Thus the poset $(\R,<_{\phi_p})$ has at most five minimal elements. This shows that all vertices of $Q$ have degree at most five.

For the moreover statement, let $p_0$ be a point on the outside of all frames, and let $o=\phi_{p_0}\in Q$. We have already seen that $o$ has degree two. For a vertex $\phi_p\in Q$, the out-degree is equal to the number of neighbours $\phi_q$ of $\phi_p$ such that if $R\in\cal R'$ is the wall separating $\phi_p$ from $\phi_q$, then $p\in R^-$ and $q\in R^+$. It follows from our above description of the neighbours of $\phi_p$ that there are at most three such neighbours; see Figure~\ref{fig:degree}.
\end{proof}

We can now prove the following, which is close to \cref{thm:bad_square_complex}, except we have not yet embedded the square complex into $\R^3$.
See \cref{sec:frames:1} for an explicit construction of a sequence of collections of frames that can be used in the below proof.

\begin{theorem} \label{thm:burbox}
There is a CAT(0) square complex $Q$ with a vertex $o\in Q$ such that $Q$ has maximum degree at most five, $Q_o$ has maximum out-degree at most three, and $Q$ cannot be isometrically embedded into any finite product of trees.
\end{theorem}

\begin{proof}
For each $k$, let $B_k$ be a (finite, connected) Burling graph with chromatic number greater than $k$ \cite{burling:oncolouring}. By \cite{pournajafitrotignon:burling:1}, $B_k$ is the intersection graph of  a finite, strict collection $\cal R_k$ of frames. Let $\cal D_k$ be as in \cref{lem:firm}, and let $Q_k$ be the CAT(0) square complex provided by Item~\ref{sh:construction} for $\cal R_k\cup\cal D_k$. The crossing graph of $Q_k$ contains $B_k$ as a subgraph, so $Q_k$ cannot be isometrically embedded into any product of $k$ trees. By \cref{lem:burdeg}, $Q_k$ has maximum degree five.

Let $o_k\in Q_k$ be the vertex provided by \cref{lem:burdeg}. Consider the infinite path graph $\N$ of positive integers. Let $Q$ be the CAT(0) square complex obtained by gluing, for all $k$, the point $o_k\in Q_k$ to the point $k\in\N$. Since every $Q_k$ isometrically embeds into $Q$, the latter cannot be isometrically embedded into any finite product of trees. Moreover, since every $Q_k$ has maximum degree five and every $o_k$ has degree at most two, $Q$ has maximum degree at most five.

Finally, let $o\in Q$ be the point $1\in\N$, or equivalently $o=o_1\in Q_1$. Since the maximum out-degree of $(Q_k)_{o_k}$ is at most three and $o_k\in Q_k$ has degree two for all $k$, the maximum out-degree of $Q_o$ is also three.
\end{proof}

\section{Square complexes with maximum degree five} \label{sec:R3}

In this section we prove that every CAT(0) square complex with maximum degree at most five admits a topological embedding in $\R^3$. In fact, we shall prove the stronger statement \cref{thm:5_in_3-space}. 

Recall that $K_{1,n}$ denotes the complete bipartite graph on one and $n$ vertices. By a \emph{page} of $K_{1,n}\times\R$, we mean a subset of the form $e\times\R$, where $e$ is an edge of $K_{1,n}$.

\begin{lemma} \label{lem:tree_gluing}
Let $n\ge 2$ be an integer. Let $Q$ be a simply connected topological space. Let $A$ be a discrete collection of subsets of $Q$, each of which is homeomorphic to $[0,1]$ and disconnects $Q$. Let $\cal Q$ be a collection of topological spaces obtained by disconnecting into two components along every element of $A$. Assume that for each $a\in A$, the disconnection $Q=Q'\cup_aQ''$ is such that, after relabelling, $a$ is contained in the boundary of $Q'$ and also has a planar open neighbourhood in $Q'$. If every $Q'\in\cal Q$ admits a topological embedding in $K_{1,n}\times\R$, then $Q$ admits a topological embedding in $K_{1,n+1}\times\R$.
\end{lemma}

\begin{proof}
Given $a\in A$, there are exactly two elements $Q',Q''\in\cal Q$ that contain paths coming from $a$ in the decomposition. We refer to these paths as being \emph{induced} by $a$. Given $Q'\in\cal Q$, let $B(Q')$ be the set of paths in $Q'$ induced by elements of $A$, each of which is homeomorphic to $[0,1]$. Fix an orientation of each $a\in A$. This induces an orientation of each $b\in B(Q')$ for each $Q'\in\cal Q$. 

By assumption, if $Q'\in\cal Q$, then we can embed it in $K_{1,n}\times\R$. We can arrange for every element of $B$ to lie along the central axis of $K_{1,n}\times\R$. Moreover, if $b\in B(Q')$ is contained in the boundary of $Q'$, then we can arrange that the central axis of $K_{1,n}\times\R$ has an open neighbourhood of $b$ that intersects $Q'$ in precisely $b$. We call such an embedding \emph{good}. We refer to an element of $B(Q')$ as being \emph{positively oriented} by the embedding if its orientation agrees with the orientation of $\R$ given by the total order. By reflecting the embedding, we can invert the orientations of all elements of $B(Q')$. See Figure~\ref{fig:arranging}.

\begin{figure}[ht]
\centering\includegraphics[height=55mm, trim = 0 3mm 0 3mm]{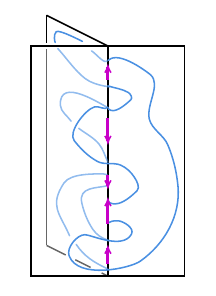}
\caption{Aligning $B(Q')$, in pink, along the central axis of $K_{1,n}\times\R$, with orientations indicated, in the proof of \cref{lem:tree_gluing}.} \label{fig:arranging}
\end{figure}

Consider the graph $T$ with a vertex for each element of $\cal Q$ and an edge from $Q'\in\cal Q$ to $Q''\in\cal Q$ whenever they contain paths induced by the same element of $A$. The graph $T$ keeps track of the cutting that we did to obtain $\cal Q$. Since $Q$ is simply connected, $T$ is a tree. Fix an element $Q_0\in\cal Q$, which we view as a root of $T$.

We shall inductively define a lexicographic order $\prec$ on the vertices of $T$, and for each $Q'\in\cal Q$ make a choice of a good embedding $f_{Q'}:Q'\to K_{1,n}\times\R$. For a nonnegative integer $n$, write $T_n$ for the set of vertices of $T$ at distance exactly $n$ from the root $Q_0$. 

Make an arbitrary choice of $b_{Q_0}\in B(Q_0)$. Let $f_{Q_0}:Q_0\to K_{1,n}\times\R$ be a good embedding. The elements of $B(Q_0)$ are ordered according to their position in the central axis after applying $f_{Q_0}$. Write $\le$ for this order, since it is induced by $\R$. Each element $b\in B(Q_0)$ corresponds to some vertex $Q_b\in T_1$. Define $\prec$ transitively on $T_1\cup\{Q_0\}$ as follows; see Figure~\ref{fig:ordering}. For $b_1,b_2\in B(Q_0)$:
\begin{itemize}
\item   if $b_1<b_{Q_0}$ then $Q_{b_1}\prec Q_{b_0}$;
\item   if $b_1>b_{Q_0}$ then $Q_{b_0}\prec Q_{b_1}$;
\item   if $b_1\le b_2<b_{Q_0}$ or if $b_{Q_0}\le b_1\le b_2$, then $Q_{b_2}\prec Q_{b_1}$.
\end{itemize}
Also declare that if $Q_1,Q_2\in T_1$ satisfy $Q_1\prec Q_2$, then $Q_1'\prec Q_2'$ for every pair of descendants in the rooted tree $T$ of $Q_1'$ and $Q_2'$, respectively. 

\begin{figure}[ht]
\centering\includegraphics[height=45mm, trim = 0 3mm 0 3mm]{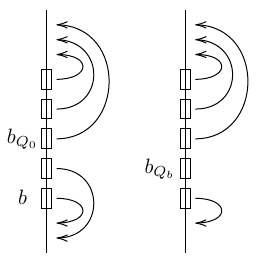}
\caption{The order $\prec$ in \cref{lem:tree_gluing} is like unfolding a pamphlet. \\ Left: defining $\prec$ on $T_1\cup\{Q_0\}$. Right: defining $\prec$ on $Q_b$ and its children.} \label{fig:ordering}
\end{figure}

Given $Q_b\in T_1$, there is a path $b_{Q_b}\in B(Q_b)$ that is induced by the same element of $A$ as $b\in B(Q_0)$. Let $f_{Q_b}:Q_b\to K_{1,n}\times\R$ be a good embedding such that the orientations of $f_{Q_0}(b)$ and $f_{Q_b}(b_0)$ in $\R$ are opposite to one another. Similarly to before, the elements of $B(Q_b)$ are ordered according to their position in the central axis of $K_{1,n}\times\R$ after applying $f_{Q_b}$, and each corresponds to a vertex of $T_2$ that is a child of $Q_b$. 

Since $Q_b$ is the root of the subtree of $T$ consisting of $Q_b$ and all its descendants, we can repeat the same construction as for $Q_0$, with $b_{Q_b}\in B(Q_b)$ playing the role of $b_{Q_0}\in B(Q_0)$, and using only the children of $Q_b$ (that is, we exclude the parent $Q_0$ of $Q_b$). This extends the definition of $\prec$ to the children of $Q_b$. Doing this for every $Q'\in T_1$, we have now defined $\prec$ on $\bigcup_{i\le2}T_i$. By iterating this process, we obtain the desired lexicographic order $\prec$ on $T$, and a choice of good embedding $f_{Q'}:Q'\to K_{1,n}\times\R$ for each $Q'\in\cal Q$.

We now use the data of $\prec$ and $\{f_{Q'}\,:\,Q\in\cal Q\}$ to define an embedding $\bigsqcup_{Q'\in\cal Q}Q' \to K_{1,n}\times\R$. Since $\prec$ is lexicographic, there is a topological embedding $g:T\to[0,1]\times\R$ such that each vertex lies in $\{0\}\times\R$. In particular, for each vertex $Q'\in T$ there is an open interval $I'(Q')\subset\R$ centred on $g(Q')$ such that if $Q''\ne Q'$, then $g(Q'')\notin I'(Q')$. Let $I(Q')\subset\R$ be the open middle third of $I'(Q')$. It contains $g(Q')$, and the closures of the intervals $I(Q')$ for $Q'\in\cal Q$ are pairwise disjoint.

For each $Q'\in\cal Q$, we have a fixed good embedding $f_{Q'}:Q'\to K_{1,n}\times\R$. By applying a homeomorphism, we can assume that its image is contained in $K_{1,n}\times I(Q')$. Since the closures of the intervals $I(Q')$ are pairwise disjoint, the disjoint union $f=\bigsqcup_{Q'\in\cal Q}f_{Q'}$ gives a topological embedding $f:\bigsqcup_{Q'\in\cal Q}Q'\to K_{1,n}\times\R$, and the images of the $Q'$ appear in the same order as given by $\prec$.

Having used $f$ to align the spaces $Q'\in\cal Q$ inside $K_{1,n}\times\R\subset K_{1,n+1}\times\R$, we shall re-glue to obtain an embedding of $Q$, following the pattern of $g(T)$. 

Let $a\in A$, and let $Q'$ and $Q''$ be the two elements of $\cal Q$ containing paths $a'\in B(Q')$ and $a''\in B(Q'')$ induced by $a$. By assumption, after relabelling we have that $a'$ is contained in the boundary of $Q'$ and has a planar open neighbourhood in $Q'$. We can therefore use the $(n+1)^\mathrm{th}$ page of $K_{1,n+1}\times\R$ to connect $a'$ to $a''$. Since the orientations of $a'$ and $a''$ are opposite, this re-gluing agrees with the original disconnection of $Q$ along $a$. See Figure~\ref{fig:type_1_regluing}. Moreover, this re-gluing constructs a thickened copy of the corresponding edge of $g(T)$ inside $[0,1]\times\R$, which is identified with the $(n+1)^\mathrm{th}$ page of $K_{1,n+1}\times\R$. Since $g$ is an embedding, no two such re-gluings conflict with one another. Thus if we do this re-gluing for every edge of $T$ then we obtain an embedding of $Q$ in $K_{1,n+1}\times\R$.
\end{proof}

\begin{figure}[ht]
\centering\includegraphics[height=60mm, trim = 0 5mm 0 3mm]{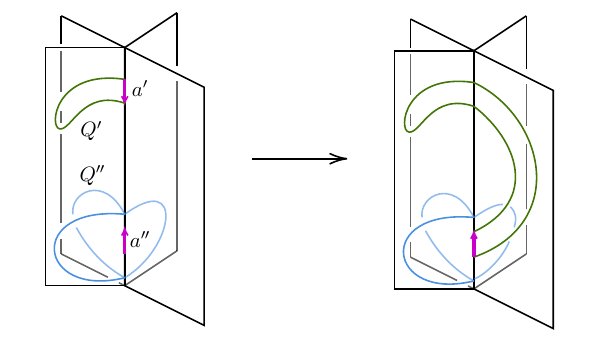}
\caption{Using the $(n+1)^\mathrm{th}$ page of $K_{1,n+1}\times\R$ to re-glue $Q'$ and $Q''$ in the proof of \cref{lem:tree_gluing}.} \label{fig:type_1_regluing}
\end{figure}

If $Q$ is a CAT(0) square complex, then an edge $e\subset Q$ is a \emph{boundary edge} if it is contained in at most one square. 

\begin{proposition} \label{prop:tripod_times_R}
Let $Q$ be a CAT(0) square complex whose vertices all have degree at most five. If no edge of $Q$ is contained in four squares, then $Q$ admits a topological embedding in $K_{1,3}\times\R$.
\end{proposition}

\begin{proof}
Let $B$ denote the set of edges of $Q$ that lie in three squares. First we show that $B$ is a disjoint union of paths. Suppose that $v\in Q$ is a vertex that is contained in at least two edges $e_1,e_2\in B$. Let $s_1,s_2,s_3$ and $t_1,t_2,t_3$ be squares witnessing that $e_1\in B$ and $e_2\in B$, respectively. The edge $e_1$ and the three squares $s_1,s_2,s_3$ provide four neighbours of $v$. Note that $e_2$ cannot be an edge of any $s_i$. Indeed, if $e_2$ were an edge of, say, $s_1$, then since the degree of $v$ is at most five, some $t_i$ would have to share an edge with either $s_2$ or $s_3$, contradicting the CAT(0) condition. Thus the other vertex of $e_2$ is the fifth and final neighbour of $v$. This shows that no edge of $B$ other than $e_1$ and $e_2$ can contain $v$. Hence each connected component of $B$ is either a path or a cycle.

Our description of the neighbours of $v$ moreover shows that, up to relabelling, $s_i$ and $t_i$ must share an edge for each $i$. Supposing that $e_1\cup e_2$ is contained in a connected component of $B$ that is a cycle, let $h$ be the hyperplane of $Q$ that intersects $s_1$ in an interval parallel to $e_1$. Since $s_1$ and $t_1$ share an edge, $h$ also intersects $t_1$ in an interval parallel to $e_2$. By following $h$ along the cycle (up to three times), we find a nontrivial loop in $h$. But $h$ is a tree, because it is a hyperplane in a CAT(0) square complex. This is a contradiction, so $B$ is a disjoint union of paths.

\mk

Now that we know that $B$ is a disjoint union of paths, we aim to apply an ``ungluing'' operation to $Q$ along $B$. Before describing this, we provide a little more information about the connected components of $B$. 

Let $E$ be a connected component of $B$. Each edge of $E$ lies in three squares and, as shown above, the squares for adjacent edges of $E$ are attached along edges adjacent to $E$, forming a $K_{1,3}\times[0,|E|]$ subcomplex. Let $v$ and $w$ be the endpoints of $E$, and let $S_1,S_2,S_3$ be the three pages of this subcomplex: each $S_i$ is a chain of $|E|$ squares attached end-to-end. Let $f_i$ be the edge of $S_i\ssm E$ that contains $v$, and let $g_i$ be the edge that contains $w$. By maximality of $E$, at least one of the $f_i$ is a boundary edge, and at least one of the $g_i$ is a boundary edge. After relabelling, there are two possibilities: either $f_1$ and $g_2$ are boundary edges, or $f_1$ and $g_1$ are boundary edges (or both). See Figure~\ref{fig:boundary}. Fix $j\in\{1,2\}$ such that $g_j$ is a boundary edge. Having made this choice for $E$, we shall refer to $E$ as being of \emph{Type $j$}.

\begin{figure}[ht]
\centering\includegraphics[height=60mm, trim = 0 3mm 0 3mm]{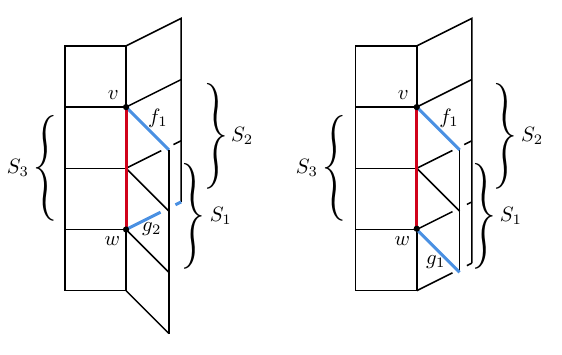}
\caption{A connected component of $B$, with endpoints $v$ and $w$, and the incident squares. On the left $E$ is Type 2, and on the right $E$ is Type 1.} \label{fig:boundary}
\end{figure}

If $E$ is of Type 1 then we detach $S_1$ from $E$, decomposing $Q$ into two connected components $Q'$ and $Q_1$, where $S_1\subset Q_1$. There is an oriented path $E_1$ in the boundary of $Q_1$ corresponding to $E$, and there is an oriented path $E'$ in the interior of $Q'$ corresponding to $E$, such that if we glue $E_1$ to $E'$, preserving orientation, then we get back the original complex $Q$. The page $S_1$ is a planar neighbourhood of $E_1$ inside $Q_1$. Note that $Q'$ and $Q_1$ are simply connected by van Kampen's theorem. 

If instead $E$ is of Type 2, then we shall unglue without disconnecting $Q$; see Figure~\ref{fig:ungluing}. In this case, if we cut $Q$ along the interior of $E$, then we create three copies of $E$, each attached to $v$ and $w$. We produce a new square complex $Q'$ from this by detaching $S_1$ from $v$, obtaining a new vertex $v'$, and detaching $S_2$ from $w$, obtaining a new vertex $w'$. Note that by the choice of $f_1$ and $g_2$, both edges containing $v'$ are boundary edges, and similarly for $w'$. The three copies of $E$ now form a path $E'$ in the boundary of $Q'$ that has endpoints $v'$ and $w'$ and passes through $v$ and $w$. We refer to $E'$ as \emph{descending from $E$}. Each vertex between $v$ and $w'$ and each vertex between $w$ and $v'$ lies in exactly two squares. The square complex $Q'$ is simply connected, because collapsing the contractible subspaces $E'\subset Q'$ and $E\subset Q$ yields homeomorphic spaces. 

\begin{figure}[ht]
\centering\includegraphics[height=60mm, trim = 0 3mm 0 3mm]{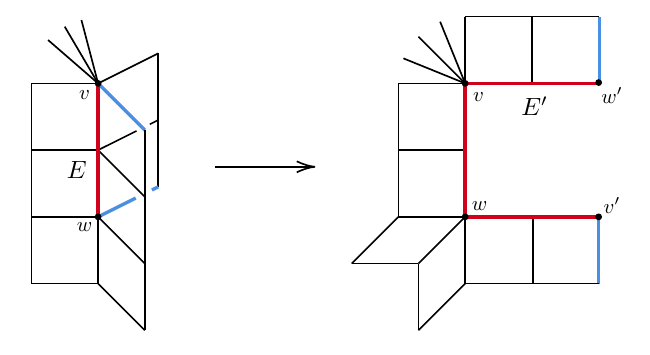}
\caption{Ungluing $Q$ along a Type 2 path $E\subset B$ with endpoints $v$ and $w$.} \label{fig:ungluing}
\end{figure}

\mk

Assign a type to each connected component of $B$. Let $B_j$ be the set of all Type~j connected components of $B$. Let $\cal Q$ be the set of cube complexes obtained by applying the above ungluing operation to all connected components of $B$. If we instead only apply the ungluing operation to just the elements of $B_2$, then we obtain a connected square complex, which we denote by $Q_2$.

Let $Q'\in\cal Q$ be one of the components obtained by completely ungluing $Q$. By construction, no edge of $Q'$ is contained in more than two squares. Thus the link of every vertex embeds as an induced subgraph of a cycle graph, so $Q'$ embeds into some surface. Since $Q'$ is simply connected, it follows that it is planar. Applying \cref{lem:tree_gluing} with $A=B_1$, we obtain a topological embedding of $Q_2$ in $K_{1,3}\times\R$.

Each $E\in B_2$ descends to a boundary component of the element of $\cal Q$ that contains it, which is planar. Moreover, $E$ is at distance at least one from each element of $B_1$. Using planarity, we can therefore arrange for the descendant $E'$ in $Q_2$ of $E$ to lie along the central axis of $K_{1,3}\times\R$, for all $E\in B_2$. Moreover, we can arrange that each such $E'$ has an open neighbourhood that intersects only a single page of $K_{1,3}\times\R$. 

We use the two other pages of $K_{1,3}\times\R$ to re-glue, as shown in Figure~\ref{fig:regluing}. Let $v'$ and $w'$ be the endpoints of $E'$, which descend from vertices $v$ and $w$ of $E$. We also view $v$ and $w$ as vertices of $E'$, as in Figure~\ref{fig:ungluing}. Let $E_1$ be the subpath of $E'$ from $w'$ to $v$, let $E_2$ be the subpath from $v$ to $w$, and let $E_3$ be the subpath from $w$ to $v'$. As noted previously, there are exactly two edges of $Q_2$ containing $v'$, and they are both boundary edges. Moreover, the vertices of $E_1$ between $v$ and $w'$ all lie in exactly two squares. We can therefore fold $E_1$ through one of the two remaining pages to glue to $E_2$ in a way that affects only the square witnessing that $E_1$ is a boundary edge. The same argument applies at $w'$, and we can fold $E_3$ through the third page to glue to $E_2$.

\begin{figure}[ht]
\centering\includegraphics[height=60mm, trim = 0 3mm 0 3mm]{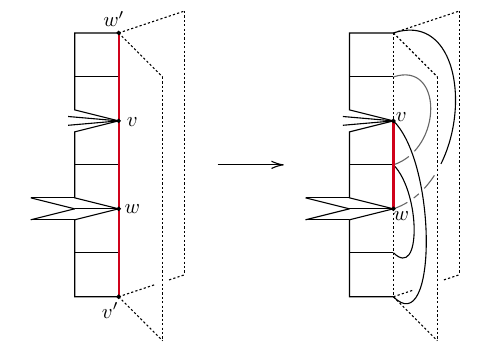}
\caption{Re-gluing a descendant of some $E\in B_2$.} \label{fig:regluing}
\end{figure}

After applying this re-gluing procedure for every Type~2 component $E\subset B$, we obtain an embedding of $Q$ in $K_{1,3}\times\R$.
\end{proof}

To prove \cref{thm:5_in_3-space} below, we shall build on \cref{prop:tripod_times_R} to include edges contained in four squares. When $Q_o$ has maximum out-degree at most three, there will no edges contained in four squares.

\begin{lemma} \label{lem:four_separates}
Let $Q$ be a CAT(0) square complex whose vertices have degree at most five. If $e$ is an edge of $Q$ that lies in four squares, then $e$ disconnects $Q$ into four connected components, and all edges adjacent to $e$ are boundary edges.
\end{lemma}

\begin{proof}
If $v$ is a vertex of $e$, then $v$ has degree five and the edges containing $v$ are precisely $e$ and an edge of each of the squares containing $e$. Thus $e$ is a cut-edge. Let $e'\ne e$ be an edge of $Q$ containing $v$, and suppose that $e'$ lies in a square $S$ not containing $e$. By our description of the edges of $Q$ containing $v$, the other edge of $S$ containing $v$ must be an edge of a square containing $e$. This violates the CAT(0) condition. 
\end{proof}


We can now prove~\cref{thm:5_in_3-space}, which we restate for convenience.



\numberedtheorem{Theorem~\ref{thm:5_in_3-space}}{}{
Let $Q$ be a CAT(0) square complex. If every vertex has degree at most five, then $Q$ can be topologically embedded into $K_{1,4}\times\R$.
Furthermore, if there also exists a vertex $o\in Q$ such that $Q_o$ has maximum out-degree at most three, then $Q$ admits a topological embedding in $K_{1,3}\times\R$.
}

\begin{proof}
Let $A$ denote the set of edges of $Q$ that lie in four squares. If $v$ is a vertex of an edge $e\in A$, then $v$ has degree five and the edges containing $v$ are precisely $e$ and an edge of each of the squares containing $e$. Thus $e$ is a cut-edge. Furthermore, if $e'\ne e$ is an edge of $Q$ containing $v$, then $e'$ lies in a unique square. Indeed, the other edge containing $v$ of a second such square would have to be among the edges described above, which would violate the CAT(0) condition. 

For each edge $e\in A$, separate the complex into two components so that $e$ is contained in the boundary of one of them. Since the edges of $Q$ that meet $e$ all lie in unique squares, $e$ has a planar open neighbourhood in that component. Doing this for each $e\in A$, we obtain a collection $\cal Q$ of CAT(0) square complexes whose vertices have degree at most five and such that no edge of any $Q'\in\cal Q$ is contained in four squares. By \cref{prop:tripod_times_R}, each $Q'\in\cal Q$ admits a topological embedding in $K_{1,3}\times\R$, and we deduce from \cref{lem:tree_gluing} that $Q$ admits a topological embedding in $K_{1,4}\times\R$.

Suppose now that there also exists a vertex $o\in Q$ such that $Q_o$ has maximum out-degree at most three.
By \cref{prop:tripod_times_R}, it is enough to show that no edge of $Q$ lies in four squares.
Suppose for sake of contradiction that there is such an edge $e=uv$, with $u$ closer to $o$ than $v$.
Both $u$ and $v$ must have degree exactly five as they have degree at most five and $uv$ is contained in four squares.
Furthermore, these four squares are in different connected components of $Q\ssm uv$, by \cref{lem:four_separates}.
Therefore, $u$ has in-degree at most one and hence out-degree at least four, a contraction.
\end{proof}

Theorem~\ref{thm:bad_square_complex} now follows from \cref{thm:burbox} and \cref{thm:5_in_3-space}.

\section{Burling-excluded cube complexes in \texorpdfstring{$\R^3$}{R3}}\label{sec:string}

In this short section we prove \cref{thm:cubeR3}.
We will show that the finite induced subgraphs of the crossing graph of a CAT(0) cube complex embedded in $\R^3$ are string graphs.
A \emph{string graph} is the intersection graph of a collection of path-connected subsets of the plane. We will then use the following theorem.


\begin{theorem}[\cite{abrishamibrianskidaviesdumasarikovarzazewskiwalczak:burling}] \label{thm:k4B_chromatic}
For every Burling graph $B$, there exists a positive integer $k$ such that every string graph excluding $B$ and $K_4$ as induced subgraphs is $k$-colourable.
\end{theorem}


\numberedtheorem{Theorem~\ref{thm:cubeR3}}{}{
For every Burling graph $B$ there exists a positive integer $k$ such that the following holds. If $Q$ is a CAT(0) cube complex that topologically embeds into $\mathbb{R}^3$ and whose crossing graph excludes $B$ as an induced subgraph, then $Q$ isometrically embeds into a product of $k$ trees.
}

\begin{proof}
Let $Q$ be a CAT(0) cube complex with a topological embedding $f:Q\to\R^3$. We have $\dim Q\le 3$, so the crossing graph $\Cross(Q)$ has no induced $K_4$ subgraph. Given a finite induced subgraph $F$ of $\Cross(Q)$, let $A\subset Q$ be a finite subset realising all consistent orientations of the hyperplanes corresponding to the elements of $F$. The convex subcomplex $Q'=\hull A\subset Q$ is finite, and $F$ is an induced subgraph of $\Cross(Q')$.

The embedding of $Q$ into $\R^3$ restricts to an embedding of $Q'$. Since $Q'$ is compact and contractible, there is some $\eps>0$ such that the $\eps$--neighbourhood $N$ of $f(Q')$ in $\R^3$ is homeomorphic to a 3--ball. By decreasing $\eps$, we can assume that if $h$ and $h'$ are two disjoint hyperplanes of $Q'$, then $\eps<\frac12\dist(f(h),f(h'))$. 

Take a small thickening of $f(h)$ for each hyperplane $h$ of $Q'$ to obtain a collection of topological circles $C_h$ in $\partial N\cong S^2$. We claim that the intersection graph of these circles is $\Cross(Q')$. Firstly, the choice of $\eps$ ensures that if $h$ and $h'$ do not cross then $C_h$ and $C_{h'}$ do not cross. Suppose instead that $h$ and $h'$ cross. We can view $h\cap h'$ as a hyperplane of the codimension-1 CAT(0) cube complex $h$, and hence it is itself a CAT(0) cube complex of dimension at most one. That is, $h\cap h'$ is a tree. By considering a neighbourhood of a leaf of this tree, we see that $C_h$ and $C_{h'}$ must cross. 

We have shown that $\Cross(Q')$ is a string graph. Since $\Cross(Q')\subset \Cross(Q)$ has no $K_4$ or $B$ induced subgraphs, the chromatic number of $\Cross(Q)$, and hence of $F$, is bounded by the number $k$ from \cref{thm:k4B_chromatic}. Since the chromatic number of $\Cross(Q)$ is the supremum of the chromatic numbers of its finite induced subgraphs \cite{debruijnerdos:colour}, this shows that $\Cross(Q)$ is $k$--colourable, and hence $Q$ isometrically embeds into a product of $k$ trees.
\end{proof}




\section{A Burling-free construction}\label{sec:noBur}

We shall use a sparse version of the Hales--Jewett theorem \cite{halesjewett:regularity} due to Prömel--Voigt~\cite{promelvoigt:sparse:grahamrothschild}.
We remark that a non-induced version would also be sufficient for our purposes and can easily be deduced from results of R{\"o}dl \cite{rodl:onramsey}. The statement uses the following terminology.

\begin{definition}[Cycle, combinatorial line]
A \emph{cycle} of length $\ell\geq 2$ on a set $X$ is a tuple $(T_1,\ldots,T_\ell)$ of distinct subsets of $X$ such that there exist distinct elements $x_1,\ldots,x_\ell\in X$ with $x_i\in T_i\cap T_{i+1}$ for $i\in[\ell-1]$ and $x_\ell\in T_\ell\cap T_1$.

For $n\in\mathbb{N}$ and a finite set $L$, a subset $C$ of the $n$-dimensional $L$-cube $L^n$ is called a \emph{combinatorial line} if there exists a non-empty set of indices $I=\{i_1,\ldots,i_k\}\subseteq[n]$ with $i_1 < \cdots < i_k$ and a choice of $x^*_i\in L$ for each $i\in L\setminus I$ such that
\[C=\bigl\{(x_1,\ldots,x_n)\in L^n \,:\, x_{i_1}=\cdots=x_{i_k}\text{ and }x_i=x^*_i\text{ for }i\notin I\bigr\}.\]
The indices in $I$ are called the \emph{active coordinates} of $C$. For $x\in L$, we let $C(x)$ be the point of $C$ with all active coordinates equal to $x$.
\end{definition}

\begin{theorem}[\cite{promelvoigt:sparse:grahamrothschild}]\label{Hales}
For any $g,k\in\mathbb{N}$ and finite set $L$ with $|L|\geq 3$, there exists an $n\in\mathbb{N}$ and a set $H\subseteq L^n$ such that every $k$-colouring of $H$ contains a monochromatic combinatorial line of $L^n$ and no tuple of fewer than $g$ combinatorial lines of $L^n$ contained in $H$ forms a cycle.
\end{theorem}

Let $L(T)$ denote the set of leaves of a tree $T$ or rooted tree $(T,r)$.

\begin{definition}[Limbs]
Let $(T,r)$ be a rooted tree. A path $P\subset T$ joining some $u \in L(T)$ to some ancestor $v$ of $u$ is called a \emph{limb} of $(T, r)$. We call
$u$ the \emph{leaf} of $P$ and $v$ the \emph{start} of $P$, and write $s(P)=v$.
We say that $Q=P\times I$  is a \emph{thick limb} of $(T,r)$ if $P$ is a limb of $(T,r)$ and $I\subset\R$ is a closed interval.
\end{definition}

Note that two thick limbs $P_1\times I_1$ and $P_2\times I_2$ can be disjoint even when $P_1$ and $P_2$ are not. We shall be interested in collections of thick limbs where we have control on how any two can intersect. Two sets are said to \emph{overlap} if they have nonempty intersection and neither is a subset of the other.

\begin{definition}[Restricted]
We say that a thick limb $P_x \times I_x$ \emph{grazes} a thick limb $P_y \times I_y$ if $I_x \subset I_{y}$, $s(P_x) \ne s(P_y)$, and $P_x \cap P_{y}=\{s(P_y)\}$.
A finite collection $\mathcal{Q}=\{P_x \times I_x : x \in X \}$ of thick limbs is said to be \emph{restricted} if:
\begin{itemize}
    \item no $P_x \times I_x$ is contained in any $P_y \times I_y$;
    \item the intervals $\{I_x :x \in X\}$ are pairwise non-overlapping and their endpoints are all distinct;
    \item if $P_x \times I_x$ and $P_{y} \times I_{y}$ intersect then one grazes the other;
    \item if $I_x \subset I_{y}$ then $s(P_x)$ is not a descendant of $s(P_y)$.
\end{itemize}
\end{definition}

See Figure~\ref{fig:cycle} for an illustration of a restricted collection of thick limbs whose overlap graph is an $8$-cycle.
Let us remark that by the definition of grazing, two thick limbs in a restricted collection of thick limbs overlap if and only if they intersect, and therefore the overlap graph and the intersection graph of a restricted collection of thick limbs are equal.

We are now ready to prove the following, which is a step towards proving Theorem \ref{thm:bad_square_complex2} and also implies and strengthens Theorem \ref{thm:chordalinterval}.
We remark that it actually implies the strengthening of Theorem \ref{thm:chordalinterval} for intersections of trivially perfect graphs and rooted directed path graphs (see \cite{golumbic1978trivially,monma1986intersection}), which are subclasses of interval and chordal graphs respectively.
A rooted tree $(T,r)$ is \emph{subcubic} if the tree $T$ has maximum degree at most $3$ and $r$ has degree at most $2$.

\begin{theorem}\label{thm:treerestricted}
    For any $g,k\in\mathbb{N}$, there exists a subcubic rooted tree $(T,r)$ and a restricted collection of thick limbs $\mathcal{Q}$ such that the overlap graph of $\mathcal{Q}$ has chromatic number at least $k$ and girth at least $g$.
\end{theorem}

\begin{proof}
    It is straightforward to construct restricted collections of thick limbs whose overlap graph is an arbitrarily long even cycle; see Figure~\ref{fig:cycle}. This establishes the theorem for $k\le 2$ and all $g$, so we proceed inductively on $k$. 
    
    \begin{figure}[ht]
    \centering \includegraphics[height=4.5cm, trim = 0 3mm 0 3mm]{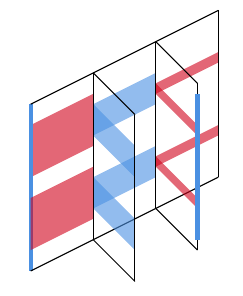}
    \caption{A restricted collection of thick limbs whose overlap graph is an 8--cycle, alternating red and blue.} \label{fig:cycle}
    \end{figure}
    
    Let $(T,r)$ be a subcubic rooted tree and $\mathcal{Q}=\{P_\ell \times I_\ell : \ell \in L \}$ be as in the theorem for $g$ and $k$. We further choose it so that $|L|\ge 3$ (this is just a technicality to allow for the use of Theorem \ref{Hales}).
    By possibly modifying the rooted tree and extending some limbs, we may further choose $(T,r)$ and $\mathcal{Q}=\{P_\ell \times I_\ell : \ell \in L \}$ so that each $P_\ell$ has positive length, $L=L(T)$ and that for each $\ell\in L$, we have that $\ell\in P_\ell$ and $\ell\not\in P_{\ell'}$ for all $\ell' \in L\backslash \{\ell\}$.

    By Theorem \ref{Hales}, there exists some positive integer $n$ and some $H\subseteq L^n$ such that every $k$-colouring of $H$ contains a monochromatic combinatorial line of $L^n$ and no tuple of fewer than $g$ combinatorial lines of $L^n$ contained in $H$ forms a cycle.
    Let $\mathcal{C}$ be the set of combinatorial lines of $L^n$ that are contained in $H$.

    We begin by constructing a new subcubic rooted tree $(T^*,r^*)$.
    For each $w\in \bigcup_{i=0}^{n-1} L^i$, let $(T_w,r_w)$ be a copy of $(T,r)$.
    For each $w\in \bigcup_{i=1}^{n-1} L^i$, let $e(w)\in L$ be the last coordinate and let $w^-$ be obtained from $w$ by removing the last coordinate.
    Now, for each $w\in \bigcup_{i=1}^{n-1} L^i$, we identify the root $r_w$ of $T_w$ with the leaf of $T_{w^-}$ corresponding to $e(w)$. We let $T^*$ be the resulting tree and let $r^*=r_\emptyset$.

    Next we describe a collection $\cal Q^*$ of thick limbs, which will come in two types. The first type are simply large vertical lines over the leaves of $T^*$. For each $w\in L^n$, we let $\ell^*_w$ be the leaf of $T_{w^-}$ (and thus of $T^*)$ corresponding to $e(w)$.
    Then $L(T^*)=\{\ell^*_w : w\in L^n\}$.
    For each $w\in L^n$, choose an interval $I_w$ with $[0,1] \subseteq I_w$ so that the intervals $\{I_w : w \in L^n\}$ are non-overlapping and their endpoints are distinct.
    For each $w\in L^n$, let
    $P_w=\ell^*_w$ be a single-vertex path. We obtain thick limbs $P_w\times I_w$.

    The second type of thick limb is less degenerate. We refer to Figure~\ref{fig:limbs} for a schematic. For each $C\in \mathcal{C}$, let $I_C$ be some subinterval of $[0,1]$, chosen so that each pair of such intervals are disjoint.
    For each $C\in \mathcal{C}$, let $\{I_{\ell_C}: \ell\in L\}$ be a homothetic copy of the intervals $\{I_\ell : \ell \in L\}$, scaled and translated so that $I_{\ell_C} \subseteq I_C$ for each $\ell \in L_C$.
    For each $C\in \mathcal{C}$, let $w_C$ be obtained from $C$ by only taking the coordinates before the first active coordinate.
    For each $\ell\in L$, let $P_{\ell_C}$ be the limb of $T^*$ with leaf $\ell^*_{C(\ell)}$ and start being the vertex of $T_{w_C}$ corresponding to the start of $P_\ell$.
    Note that $P_{\ell_C}$ contains the leaf of $T_{w_C}$ corresponding to $\ell$.
    For each $C\in \cal C$, let $\mathcal{Q}_C = \{P_{\ell_C} \times I_{\ell_C} : \ell \in L\}$.
    
    \begin{figure}[ht]
    \centering \includegraphics[height=5cm, trim = 0 3mm 0 3mm]{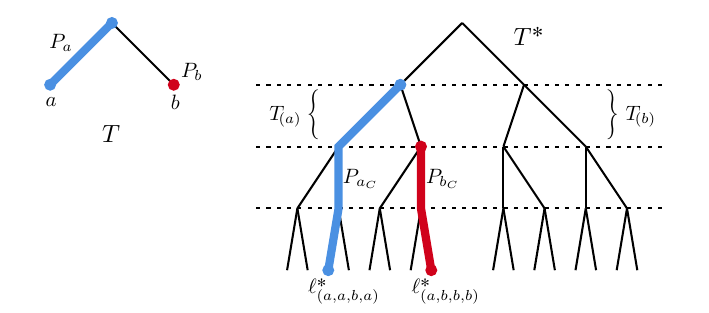}
    \begin{center} \makebox[0pt]{\begin{minipage}{1.1\textwidth}
    \caption{Simplified example of the construction of the limbs $P_{\ell_C}$ for a combinatorial line $C\in\cal C$ and leaves $\ell\in L$. Left: $T$, with two leaves $L=\{a,b\}$ and two limbs, $P_a$ in blue and $P_b$ in red. Right: $T^*$, with $n=4$. For $C=(a,*,b,*)$ we have $w_C=(a)$, so $T_{w_C}$ is the left copy of $T$ on the second level of $T^*$. Thus $P_{a_C}$, in blue, starts in $T_{(a)}$ and has leaf $\ell^*_{C(a)}=\ell^*_{(a,a,b,a)}$, whereas $P_{b_C}$ starts in $T_{(a)}$ and has leaf $\ell^*_{C(b)}=\ell^*_{(a,b,b,b)}$. } \label{fig:limbs}
    \end{minipage}}\end{center}
    \end{figure}

    Set $\mathcal{Q}^*= \{P_w \times I_w \,:\, w \in L^n\} \cup \bigcup_{C\in \mathcal{C}} \mathcal{Q}_C$.
    We shall show that $\mathcal{Q}^*$ is the desired restricted collection of thick limbs.  


    By the inductive hypothesis, for each $C\in \mathcal{C}$, we have that the overlap graph of $\mathcal{Q}_C$ has chromatic number at least $k-1$, girth at least $g$, and that $\mathcal{Q}_C$ is a restricted collection of thick limbs of $(T^*,r^*)$. Indeed, the overlaps take place in $T_{w_C}\times\R$.
    For distinct $C,C'\in \mathcal{C}$, the thick limbs of $\mathcal{Q}_C$ are disjoint from those of $\mathcal{Q}_{C'}$ since $I_C$ and $I_{C'}$ are disjoint.
    It is easy to see that $\cal Q^*$ is a restricted collection of thick limbs, because $I_C\subseteq [0,1]$ for all $C \in \mathcal{C}$, and for all $w\in L^n$ we have $[0,1]\subseteq I_w$ and $P_w$ is just a leaf of $T^*$.
    It remains to show that the overlap graph of $\mathcal{Q}^*$ has chromatic number at least $k+1$ and girth at least $g$.

    It is clear that the thick limbs in $\{P_w \times I_w : w\in L^n\}$ are disjoint.
    Also, as discussed, for distinct $C,C'\in \mathcal{C}$, we have that the members of $\mathcal{Q}_C$ are disjoint from those of $\mathcal{Q}_{C'}$. Hence, for each $C\in \mathcal{C}$, $\ell\in L$, the only element of $\mathcal{Q}^*\backslash \mathcal{Q}_C$ that $P_{\ell_C} \times I_{\ell_C}$ can intersect is $P_{C(\ell)}\times I_{C(\ell)}$, and the intersection is indeed nonempty because $I_{\ell_C}\subset[0,1]\subset I_{C(\ell)}$.
    In particular, they overlap since $P_{\ell_C}$ has positive length.
    
    Consider some cycle $F$ of the overlap graph of $\mathcal{Q}^*$ that is not contained in some $\mathcal{Q}_C$.
    As no tuple of fewer than $g$ elements of $\mathcal{C}$ form a cycle, it follows that $F$ has length at least $2g$, since it must go in and out of at least $g$ of the collections $\mathcal{Q}_C$.
    Otherwise, if $F$ is contained in some $\mathcal{Q}_C$, then by the inductive hypothesis, $F$ has length at least $g$.
    Therefore, the overlap graph of $\mathcal{Q}^*$ has girth at least $g$.

    Finally, suppose for sake of contradiction that there is some $k$-colouring of the overlap graph of $\mathcal{Q}^*$.
    Then, by Theorem \ref{Hales}, there is some $C\in \mathcal{C}$ such that $\{P_{C(\ell)}\times I_{C(\ell)} : \ell \in L\}$ is monochromatic.
    For each $\ell\in L$, the thick limb $P_{\ell_C}\times I_{\ell_C}$ overlaps with $P_{C(\ell)}\times I_{C(\ell)}$.
    In particular, no element of $\mathcal{Q}_C$ receives the same colour that $\{P_{C(\ell)}\times I_{C(\ell)}\}$ received.
    But then $\mathcal{Q}_C$ is $(k-1)$-coloured, contradicting the inductive hypothesis that the overlap graph of $\mathcal{Q}$ has chromatic number at least $k$.
    Hence the overlap graph of $\mathcal{Q}^*$ has chromatic number at least $k+1$, as desired.
\end{proof}



As in \cref{sec:R3}, we shall use a collection of thick limbs to define a dual CAT(0) cube complex, and we need additional structure in order to control its maximum degree.

Let $(T,r)$ be a rooted tree. We write $T^\times=V(T)\ssm\{r\}$ for the set of non-root vertices. Given $v\in T$, let $T_v$ denote the subtree consisting of $v$ and all its descendants.

\begin{definition}[Constrained]
We say that $W=F\times I$ is a \emph{thick subtree} of a subcubic rooted tree $(T,r)$ if $F$ is a subtree of $T$ and $I\subset\R$ is a closed interval.
A finite collection $\mathcal{W}=\{W_x=F_x\times I_x \,:\, x\in X\}$ of thick subtrees of $(T,r)$ is \emph{constrained} if all of the following hold. See Figure~\ref{fig:constrained} for an illustration.
\begin{itemize}
    \item No three elements of $\mathcal{W}$ overlap pairwise.
    \item $T^\times \subset X$ and for each $v\in T^\times$ we have $W_v = T_v\times \R$.
    \item The intervals $\{I_x \,:\, x\in X \ssm T^\times \}$ are pairwise non-overlapping and their endpoints are all finite and distinct.
    \item For each vertex $v \in T$, if three elements $F_1\times I_1$, $F_2\times I_2$, and $F_3 \times I_3$ of $\mathcal{W}$ have $v\in F_1\cap F_2\cap F_3$ and the intervals $I_1,I_2,I_3$ are disjoint, then there exists some $F\times I \in \mathcal{W}$ such that $v\in F$ and $I$ contains exactly two of $I_1, I_2, I_3$. 
\end{itemize}
\end{definition}

\begin{figure}[ht]
\centering \includegraphics[height=3.7cm, trim = 0 3mm 0 3mm]{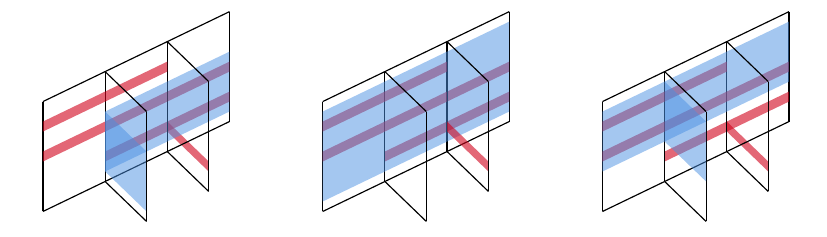}
\begin{center} \makebox[0pt]{\begin{minipage}{1.1\textwidth}
\caption{Left: the final condition in the definition of constrained. When the thick subtrees from the second condition are added, this example fails the first condition. Middle: the blue thick subtree contains exactly two of the red thick subtrees, but fails the final condition of being constrained. Right: a correct example.} \label{fig:constrained}
\end{minipage}}\end{center}
\end{figure}

Next we show that every restricted collection of thick limbs can be extended to a constrained collection of thick subtrees without increasing the chromatic number of the overlap graph much.

\begin{lemma}\label{lem:treeconstrained}
    Let $(T,r)$ be a subcubic rooted tree and let $\mathcal{Q}=\{P_x \times I_x : x\in X\}$ be a restricted collection of thick limbs.
    There exists a constrained collection of thick subtrees $\mathcal{W}$ containing $\mathcal{Q}$ such that the overlap graph of $\mathcal{W}\backslash \mathcal{Q}$ is bipartite.
\end{lemma}

\begin{proof}
    Let $\mathcal{S} = \{W_v=T_v\times\R \,:\, v\in T^\times\}$ as in the second bullet of the definition of a constrained collection of thick subtrees.
    Note that the thick subtrees in $\mathcal{S}$ are pairwise non-overlapping.
    
    Choose a finite collection of finite intervals $\{I_y : y\in Y\}$ of $\R$ such that:
    \begin{itemize}
    \item   for every $y\in Y$ there is some $x\in X$ such that $I_x\subset I_y$;
    \item   the intervals $ \{I_x : x\in X\} \cup \{I_y : y\in Y\}$ are pairwise non-overlapping with distinct endpoints;
    \item   for every disjoint triple $I_1,I_2,I_3\in \{I_x : x\in X\} \cup \{I_y : y\in Y\}$, there exists some $y\in Y$ such that $I_y$ contains exactly two of $I_1,I_2,I_3$.
    \end{itemize} 
    For each $y\in Y$, let $T_y$ be the subgraph of $T$ consisting of all limbs $P$ with start $r$ such that $P\times I_y$ contains some thick limb of $\mathcal{Q}$.
    Let $\mathcal{Y}= \{T_y\times I_y : y\in Y\}$.
    Note that the thick subtrees in $\mathcal{Y}$ are pairwise non-overlapping.
    Let $\mathcal{W}= \mathcal{Q} \cup \mathcal{S} \cup \mathcal{Y}$.
    
    The overlap graph of $\mathcal{W}\backslash \mathcal{Q} = \mathcal{S} \cup \mathcal{Y}$ is bipartite, so it remains to show that $\mathcal{W}$ is constrained. The choice of $\cal S$ gives the second bullet in the definition.
    By construction, the intervals $\{I_x : x\in X\cup Y \cup T^\times \}$ are pairwise non-overlapping and their (finite) endpoints are distinct. This gives the third bullet.

    Next, consider a vertex $v \in T$ and a triple $F_1\times I_1$, $F_2\times I_2$, $F_3 \times I_3$ in $\mathcal{W}$ with $v\in F_1\cap F_2\cap F_3$ and $I_1,I_2,I_3$ disjoint.
    Clearly $F_1\times I_1, F_2\times I_2, F_3 \times I_3 \not\in \mathcal{S}$, so $I_1,I_2,I_3$ have finite endpoints. 
    By the choice of $Y$, there exists some $T_y \times I_y \in \mathcal{Y}\subseteq \mathcal{W}$ such that $I$ contains exactly two of $I_1, I_2, I_3$, say $I_1$ and $I_2$.
    From the definition of $T_y$, both $F_1$ and $F_2$ must be subtrees of $T_y$. In particular, $v\in T_y$.

    It now just remains to show that no three $W_1,W_2,W_3\in \mathcal{W}$ pairwise overlap.
    Suppose for a contradiction that $W_1=F_1\times I_1$, $W_2=F_2 \times I_2$, and $W_3 = F_3 \times I_3$ are pairwise overlapping elements of $\cal W$.
    No pair of intervals overlap, so after relabelling we have that $I_1\subset I_2 \subset I_3$ and the subtrees $F_1,F_2,F_3$ pairwise overlap (and in particular pairwise intersect).
    No two elements of $\cal S$ overlap, so we can assume that $W_1,W_2\notin\cal S$.
    By construction, no element of $\mathcal{Y}$ overlaps with any other element of $\cal Y$, so at most one $W_i$ comes from $\cal Y$. Moreover, $F\times I\in\cal Y$ does not overlap with any $F'\times I'\in\cal Q$ with $I'\subset I$, so $W_2,W_3\notin\cal Y$.  
    This shows that $W_2\in \cal Q$.

    If $W_1\in\cal Y$, then there is some $F\times I\in\cal Q$ such that $I\subset I_1$, and hence $I\subset I_2$. 
    Since $\cal Q$ is a restricted collection of thick limbs, the start vertex $s(F)$ is not a descendant of $s(F_2)$. Because $F_1$ is the union of such limbs, the fact that $F_1$ and $F_2$ intersect implies that there is some such $F\times I$ such that $s(F_2)\in F$. 
    By again using restrictedness, any such $F\times I$ grazes $W_2$, so $F_1\cap F_2=\{s(F_2)\}$. By replacing $W_1$ with such an $F\times I$, we can therefore assume that $W_1\in\cal Q$, and we have that $F_1\cap F_2=\{s(F_2)\}$.
    By the Helly property for subtrees \cite{gyarfaslehel:helly}, it follows that $F_1\cap F_2\cap F_3 = \{s(F_2)\}$.
    
    We cannot have $W_3\in\cal S$, for then we would have $W_2\subset T_{s(F_2)}\times\R\subset W_3$.
    Thus $W_3\in \cal Q$. But now since $\cal Q$ is restricted, $W_2$ must graze $W_3$ and therefore $\{s(F_3)\}=F_2\cap F_3\supset F_1\cap F_2\cap F_3= \{s(F_2)\}$. But $s(F_2)\ne s(F_3)$ since $W_2$ and $W_3$ graze each other, which is a contradiction.
    %
    %
    Hence $\mathcal{W}$ is a constrained collection of thick subtrees, as desired.
\end{proof}

Similarly to Item~\ref{sh:construction}, a collection of thick subtrees has a dual CAT(0) cube complex.

\bsh{Construction} \label{sh:construction_BF}
Let $(T,r)$ be a tree, and let $\cal W$ be a finite collection of thick subtrees. By viewing each $W\in\cal W$ as a bipartition of $T\times\R$ into (closed) inside and (open) outside, we obtain a wallspace $(T\times\R,\cal W)$. As described in Item~\ref{sh:sageev}, this wallspace has a dual CAT(0) cube complex $Q$. Two walls in $Q$ cross if and only if the corresponding elements of $\cal W$ overlap, so the crossing graph of $Q$ is the overlap graph of $\cal W$.
\esh


    

Our next goal is to show that if we have a constrained collection of thick subtrees, then the dual CAT(0) cube complex from Item~\ref{sh:construction_BF} is 2-dimensional and has degree at most six. First we shall need the following technical lemma.

\begin{lemma} \label{lem:principal}
Let $(T,r)$ be a subcubic rooted tree, let $\mathcal{W}$ be a constrained collection of thick subtrees, and let $Q$ be the CAT(0) cube complex dual to $(T\times\R,\mathcal{W})$ as described in Item~\ref{sh:construction_BF}. Every vertex of $Q$ is represented by the principal ultrafilter $\phi_p$ corresponding to some $p\in T\times\R$.
\end{lemma}

\begin{proof}
Let $\phi\in Q$ be a vertex, which is an ultrafilter on $\cal W$. For $W\in\cal W$, let us write $W^c=T\times\R\ssm W$. For each $W\in\cal W$, the ultrafilter $\phi$ selects an element of $\{W,W^c\}$. We shall find a point $p\in T\times\R$ such that $\phi=\phi_p$. In other words, $p\in\phi(W)$ for all $W\in\cal W$.

Given $v\in T^\times$, let $e_v$ denote the edge from $v$ to its parent. If $\phi(W_v)=W_v$, then direct $e_v$ towards $v$. Suppose the directed graph $T$ has sinks $v_1$ and $v_2$. After relabelling, we can assume that $v_2$ is not a descendant of $v_1$. In particular, $v_1\in T^\times$, so since $v_1$ is a sink we must have $\phi(W_{v_1})=W_{v_1}$. Because $\phi$ is an ultrafilter, if $v_1$ is not a descendant of $v_2$, then $\phi(W_{v_2})=W_{v_2}^c$, contradicting the fact that $v_2$ is a sink. Thus if $v$ is the vertex adjacent to $v_2$ on the geodesic from $v_2$ to $v_1$, then $v$ is a child of $v_2$, and $v_1$ is a descendant of $v$. But $\phi$ is an ultrafilter, so we must have $\phi(W_v)=W_v$, which contradicts the assumption that $v_2$ is a sink.

We have shown that the finite directed graph $T$ has a unique sink $v_0$, and by definition this means that $\{v_0\}\times\R\subset\phi(W_v)$ for all $v\in T^\times$. Let $v_{-1}$ denote the parent of $v_0$, if it exists, and let $v_1$ and $v_2$ denote the children of $v_0$, if they exist. We have $\phi(W_{v_1})=W_{v_1}^c$, $\phi(W_{v_2})=W_{v_2}^c$, and $\phi(W_{v_{-1}})=W_{v_{-1}}$.

Given $W=F\times I\in\cal W$, if $v_0\notin F$, then $W$ is disjoint from $\{v_0\}\times\R$ and is nested in one of $W_{v_1}$, $W_{v_2}$, or $W_{v_{-1}})^c$. Since $\phi$ is an ultrafilter, this implies that $\phi(W)=W^c$, so $\{v_0\}\times\R\subset\phi(W)$.

Finally, consider $\cal W'=\{W=F\times I\in\cal W\,:\,v_0\in F\}$. Let $Y$ be an indexing set for $\cal W'$, so that $\cal W'=\{W_y=F_y\times I_y\,:\,y\in Y\}$, and let $\cal I'=\{I_y\,:\,y\in Y\}$. The restriction of $\phi$ to $\cal W'$ induces an ultrafilter $\phi'$ on $\cal I'$. Since no two elements of $\cal I'$ overlap, there is a minimal element $I_0\subset\R$ of the poset $(\cal I',\subset)$ such that $\phi'(I_0)=I_0$. Since the finite endpoints of the elements of $\cal I'$ are distinct, there is a point $z\in I_0$ such that if $I\in\cal I'\ssm\{I_0\}$ has $I\subset I_0$, then $z\notin I$.

Since $\phi'$ is an ultrafilter, if $I\in\cal I'$ is disjoint from $I$, then $\phi(I)\supset\phi(I_0)=I_0$, so $z\in\phi(I)$. If $I$ intersects $I_0$, then either it contains $I_0$, in which case $\phi(I)=I\ni z$, or it is nested in $I_0$, in which case the minimality of $I_0$ implies that $\phi(I)=I^c$. The choice of $z$ then ensures that $z\in\phi(I)$. 

Consider the point $p=(v_0,z)\in T\times\R$. By the previous paragraph we have $p\in\phi(W)$ for all $W\in\cal W'$. By the choice of $v_0$ we have $p\in\phi(W)$ for all $W\in\cal W\ssm\cal W'$. Thus $\phi_p=\phi$ as desired.
\end{proof}

Note that in the above proof we did not use the final item in the definition of a constrained collection of thick subtrees.

\begin{lemma}\label{lem:treedeg}
    Let $(T,r)$ be a subcubic rooted tree, let $\mathcal{W}$ be a constrained collection of thick subtrees, and let $Q$ be the CAT(0) cube complex dual to $(T\times\R,\mathcal{W})$ as described in Item~\ref{sh:construction_BF}. We have $\dim Q\le2$ and the degree of each vertex of $Q$ is at most six. Moreover, $Q$ has a vertex of degree at most four.
\end{lemma}

\begin{proof}
    The crossing graph of $Q$ is the overlap graph of $\cal W=\{W_x=F_x\times I_x\,:\,x\in X\}$, so since no three elements of $\mathcal{W}$ pairwise overlap, the dimension of $Q$ is at most two.
    By \cref{lem:degree}, to bound the maximum degree of $Q$, we must bound the number of minimal elements of $(\mathcal{W},<_\phi)$ for each $\phi\in Q$. 
    By \cref{lem:principal}, the map $T \times \R \to Q$, given by sending $p\in T \times \R$ to the principal ultrafilter $\phi_p$ it defines is onto. Thus controlling the number of minimal elements of $(\cal W,<_\phi)$ is equivalent to controlling the number of $W\in\cal W$ that bound regions of $T \times \R$ representing $\phi$. 

    Consider a point $(v,z)\in T \times \R$.
    We can bound the regions of $T\times \R$ representing $\phi_p$ by the thick subtrees:
    \begin{enumerate}
        \item $T_v\times \R$ of $\mathcal{W}$, if $v\ne r$;
        \item $T_w\times \R$ of $\mathcal{W}$ for each child $w$ of $v$;
        \item $F_1\times I_1 \in \W$, where $I_1$ is minimal subject to $v\in V(F_1)$ and $z\in I_1$;
        \item $F\times I\in \W$ such that $I$ is maximal subject to $v\in V(F)$ and $I\subsetneq I_1$.
    \end{enumerate}
    There is at most one thick subtree of type (1), depending on whether $v=r$.
    Since $T$ is subcubic, there are at most two of type (2), depending on the number of children of $v$.
    Clearly there is at most one thick subtree of type (3).
    By the forth bullet of the definition of a constrained collection of thick subtrees, there are at most two thick subtrees of type~(4).
    Therefore, the degree of each vertex of $Q$ is at most $1+2+1+2=6$. If $v$ is a leaf of $T$, then the degree of $\phi_{(v,z)}$ is at most four, because $v$ has no children.
\end{proof}

Combining the above results, we can now prove \cref{thm:bad_square_complex2}. It follows directly from the following theorem by letting $n$ be greater than the size of some 5--chromatic Burling graph.

\begin{theorem} \label{thm:6}
For every positive integer $n$ there is a CAT(0) square complex $Q$ such that the vertices of $Q$ have degree at most six, $Q$ cannot be isometrically embedded into any finite product of trees, and every subgraph of the crossing graph of $Q$ with at most $n$ vertices is 4-colourable.
\end{theorem}
    
\begin{proof}
    Fix a value of $n$. By \cref{thm:treerestricted}, there exists a subcubic rooted tree $(T_k,r_k)$ and a restricted collection of thick limbs $\cal Q_k$ such that the overlap graph of $\cal Q_k$ has chromatic number greater than $k$ and girth greater than $n$. In particular, every $n$-vertex subgraph of the overlap graph is bipartite.

    By \cref{lem:treeconstrained}, there is a constrained collection of thick subtrees $\mathcal{W}_k$ containing $\mathcal{Q}_k$ such that the overlap graph of $\mathcal{W}_k\backslash \mathcal{Q}_k$ is bipartite.
    Therefore, every $n$-vertex subgraph is $4$--colourable.
    Clearly the overlap graph still has chromatic number greater than $k$.

    Let $Q_k$ be the CAT(0) cube complex dual to $(T_k\times\R,\cal W_k)$ as described in Item~\ref{sh:construction_BF}. By \cref{lem:treedeg}, it has dimension at most two, its vertices all have degree at most six, and it has a vertex $o_k$ of degree at most four.
    The crossing graph of $Q_k$ is the overlap graph of $\cal W$, and therefore has chromatic number greater than $k$. Hence $Q_k$ cannot be isometrically embedded into a product of $k$ trees.

    Consider the infinite path graph $\N$ of positive integers. Let $Q$ be the CAT(0) square complex obtained by gluing, for all $k$, the point $o_k\in Q_k$ to the point $k\in\N$. Its vertices have degree at most six. Since every $Q_k$ isometrically embeds in $Q$, the latter cannot be isometrically embedded into any finite product of trees. The crossing graph of $Q$ is the disjoint union of the crossing graphs of $\N$ and the $Q_k$, so every $n$-vertex subgraph of it is 4--colourable.
\end{proof}


\appendix
\section{Explicit construction of frames for Burling graphs} \label{sec:frames:1}

In Section~\ref{sec:frames}, we used the result from \cite{pournajafitrotignon:burling:1} that every Burling graph is the intersection graph of a strict collection of frames in order to produce a wallspace structure on $\R^2$, ultimately leading to \cref{thm:bad_square_complex}. Unfortunately, this leaves the collection of frames somewhat mysterious.

In this appendix, we give an explicit sequence of collections of frames that can be used (after applying \cref{lem:firm}) to produce CAT(0) square complexes of degree at most five that do not isometrically embed into products of increasingly many trees. The construction is a translation of a recent alternative construction of the sequence of Burling graphs in~\cite{abrishamibrianskidaviesdumasarikovarzazewskiwalczak:burling}.

\ubsh{Base case}
The collection $C_1$ consists of two frames, $R_a$ and $R_b$. The frame $R_a$ is bounded. The frame $R_b$ is unbounded and exits $R_a$. See the left of Figure~\ref{fig:C'12}.
\uesh

\begin{figure}[ht]
\centering\includegraphics[height=4.5cm]{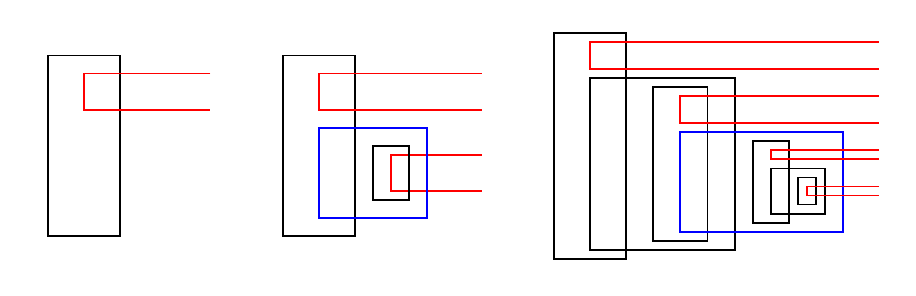}
\caption{Left: the configuration $C_1$, with its one unbounded frame coloured red. Middle: the configuration $C_2$. The blue frame is a ``duplicate'' of the red frame above it, and inside the blue frame we see a copy of $C_1$. On the right is $C_3$, with the duplicate again blue. The intersection graph of $C_3$ contains a pentagon.} \label{fig:C'12}
\end{figure}

\ubsh{The queue}
For each $n$, let $U_n\subset C_n$ denote the set of unbounded frames. We keep track of a ``queue'' $\cal Q_n$, which is a sequence of elements of $U_n$, whose length will grow as $n$ grows. We have $U_1=\{R_b\}$, and we set $\cal Q_1=(R_b)$.
\uesh

\ubsh{Inductive step}
We now describe how to produce $C_{n+1}$ from $C_n$. See the middle of Figure~\ref{fig:C'12}. Start by taking two copies $C^1_n$ and $C^2_n$ of $C_n$. Let $R\in U_n$ be the first term in the queue $\cal Q_n$. By rescaling and translating $C^2_n$, we can place it immediately below the frame $R\in C^1_n$ and outside every frame of $C^1_n$.

If $R\in C_n$ intersects a frame in $C_n$, then it exits that frame because $R$ is unbounded. We introduce a new ``duplicate'' frame $R'$ that is bounded, exits exactly the same set of frames that $R$ exits, and in which every bounded element of $C^2_n$ is nested. Observe that the frames in $C_{n+1}$ that intersect $R'$ are precisely the frames in $C^1_n$ that intersect $R$ together with all the unbounded frames in $C^2_n$.

To obtain $\cal Q_{n+1}$ from $\cal Q_n$, delete the first term of $\cal Q_n$ and then append the elements of $U_{n+1}$ in an arbitrary order.
\uesh

It is easy to see that the intersection graph of $C_n$ is the $n^\mathrm{th}$ term in the sequence of graphs constructed in \cite[\S2.1]{abrishamibrianskidaviesdumasarikovarzazewskiwalczak:burling}. Thus every Burling graph appears as an induced subgraph of the intersection graph of $C_n$ for sufficiently large $n$, by \cite[Lem.~2.2]{abrishamibrianskidaviesdumasarikovarzazewskiwalczak:burling}.

\bibliographystyle{alpha}
{\footnotesize
\bibliography{bibtex}}

\newcommand{\etalchar}[1]{$^{#1}$}
\begin{thebibliography}{DKK{\etalchar{+}}26}

\bibitem[ABCR94]{assousbouchittecharrettonrozoy:finite}
Marc~Roland Assous, Vincent Bouchitt\'e, Christine Charretton, and Brigitte Rozoy.
\newblock Finite labelling problem in event structures.
\newblock {\em Theoret. Comput. Sci.}, 123(1):9--19, 1994.

\bibitem[ABD{\etalchar{+}}25]{abrishamibrianskidaviesdumasarikovarzazewskiwalczak:burling}
Tara Abrishami, Marcin Briański, James Davies, Xiying Du, Jana Masaříková, Paweł Rzążewski, and Bartosz Walczak.
\newblock Burling graphs in graphs with large chromatic number.
\newblock {\em arXiv:2510.19650}, 2025.

\bibitem[BBP26]{baderbensaidpetyt:quasiisometric:rigidity}
Shaked Bader, Oussama Bensaid, and Harry Petyt.
\newblock Quasiisometric embeddings between right-angled {A}rtin groups: rigidity.
\newblock {\em arXiv:2605.12300}, 2026.

\bibitem[BC93]{barthelemyconstantin:median}
Jean-Pierre Barth\'el\'emy and Julien Constantin.
\newblock Median graphs, parallelism and posets.
\newblock {\em Discrete Math.}, 111(1-3):49--63, 1993.
\newblock Graph theory and combinatorics (Marseille-Luminy, 1990).

\bibitem[BCE10]{bandeltchepoieppstein:combinatorics}
Hans-J\"urgen Bandelt, Victor Chepoi, and David Eppstein.
\newblock Combinatorics and geometry of finite and infinite squaregraphs.
\newblock {\em SIAM J. Discrete Math.}, 24(4):1399--1440, 2010.

\bibitem[BCE15]{bandeltchepoieppstein:ramified}
Hans-J\"urgen Bandelt, Victor Chepoi, and David Eppstein.
\newblock Ramified rectilinear polygons: coordinatization by dendrons.
\newblock {\em Discrete Comput. Geom.}, 54(4):771--797, 2015.

\bibitem[BE51]{debruijnerdos:colour}
N.~G.~de Bruijn and P.~Erd\"os.
\newblock A colour problem for infinite graphs and a problem in the theory of relations.
\newblock {\em Indag. Math.}, 13:369--373, 1951.

\bibitem[Bow14]{bowditch:embedding}
Brian~H. Bowditch.
\newblock Embedding median algebras in products of trees.
\newblock {\em Geom. Dedicata}, 170:157--176, 2014.

\bibitem[Bow22]{bowditch:median:book}
Brian~H. Bowditch.
\newblock Median algebras.
\newblock {\em Preprint available at \mbox{bhbowditch.com/papers/median-algebras.pdf}}, 2022.

\bibitem[Bur65]{burling:oncolouring}
James~P. Burling.
\newblock {\em On colouring problems of families of polytopes}.
\newblock PhD thesis, University of Colorado, Boulder, 1965.

\bibitem[But19]{button:groups}
J~Button.
\newblock Groups acting faithfully on trees and properly on products of trees.
\newblock {\em arXiv:1910.04614}, 2019.

\bibitem[BV89]{bandeltvandevel:embedding}
H.-J. Bandelt and M.~van~de Vel.
\newblock Embedding topological median algebras in products of dendrons.
\newblock {\em Proc. London Math. Soc. (3)}, 58(3):439--453, 1989.

\bibitem[Car26]{carmesin:embedding}
Johannes Carmesin.
\newblock Embedding simply connected 2-complexes in 3-space.
\newblock {\em Mem. Amer. Math. Soc.}, 319(1625):v+82, 2026.

\bibitem[CH13]{chepoihagen:onembeddings}
Victor Chepoi and Mark Hagen.
\newblock On embeddings of {CAT}(0) cube complexes into products of trees via colouring their hyperplanes.
\newblock {\em J. Combin. Theory Ser. B}, 103(4):428--467, 2013.

\bibitem[CH26]{chepoihagen:corrigendum}
Victor Chepoi and Mark Hagen.
\newblock Corrigendum to ``{O}n embeddings of {CAT}(0) cube complexes into products of trees via colouring their hyperplanes''.
\newblock {\em J. Combin. Theory Ser. B}, 179:367--372, 2026.

\bibitem[Che00]{chepoi:graphs}
Victor Chepoi.
\newblock Graphs of some {CAT}(0) complexes.
\newblock {\em Adv. in Appl. Math.}, 24(2):125--179, 2000.

\bibitem[Che12]{chepoi:nice}
Victor Chepoi.
\newblock Nice labeling problem for event structures: a counterexample.
\newblock {\em SIAM J. Comput.}, 41(4):715--727, 2012.

\bibitem[CKS26]{chaniotiskoertsspirkl:intersections}
Aristotelis Chaniotis, Hidde Koerts, and Sophie Spirkl.
\newblock Intersections of graphs and {$\chi$}-boundedness.
\newblock {\em SIAM J. Discrete Math.}, 40(1):420--448, 2026.

\bibitem[CMS25]{chaniotismiraftabspirkl:graphs}
Aristotelis Chaniotis, Babak Miraftab, and Sophie Spirkl.
\newblock Graphs of bounded chordality.
\newblock {\em Electron. J. Combin.}, 32(4.7):1--22, 2025.

\bibitem[Dav21]{davies:box}
James Davies.
\newblock Box and segment intersection graphs with large girth and chromatic number.
\newblock {\em Adv. Comb.}, 7:1--9, 2021.

\bibitem[Des47]{descartes:three}
Blanche Descartes.
\newblock A three colour problem.
\newblock {\em Eureka}, 9(21):24--25, 1947.

\bibitem[Des54]{descartes:solution}
Blanche Descartes.
\newblock Solution to advanced problem 4526.
\newblock {\em Amer. Math. Monthly}, 61(5):352--353, 1954.

\bibitem[DHY24]{davieshatzelyepremyan:counterexample}
James Davies, Meike Hatzel, and Liana Yepremyan.
\newblock Counterexample to {B}abai's lonely colour conjecture.
\newblock {\em arXiv:2410.05199}, 2024.

\bibitem[DHY26]{davieshatzelyepremyan:minimal}
James Davies, Meike Hatzel, and Liana Yepremyan.
\newblock Minimal {C}ayley graphs with large chromatic number.
\newblock {\em arXiv:2608.06254}, 2026.

\bibitem[DJ99]{dranishnikovjanuszkiewicz:every}
A.~Dranishnikov and T.~Januszkiewicz.
\newblock Every {C}oxeter group acts amenably on a compact space.
\newblock In {\em Proceedings of the 1999 {T}opology and {D}ynamics {C}onference ({S}alt {L}ake {C}ity, {UT})}, volume~24, pages 135--141, 1999.

\bibitem[DKK{\etalchar{+}}26]{davieskellerkleistsmorodinskywalczak:solution}
James Davies, Chaya Keller, Linda Kleist, Shakhar Smorodinsky, and Bartosz Walczak.
\newblock A solution to {R}ingel's circle problem.
\newblock {\em J. Eur. Math. Soc. (JEMS)}, 28(11):4873--4892, 2026.

\bibitem[FJM{\etalchar{+}}18]{felsnerjoretmicektrotterwiechert:burling}
Stefan Felsner, Gwena\"el Joret, Piotr Micek, William~T. Trotter, and Veit Wiechert.
\newblock Burling graphs, chromatic number, and orthogonal tree-decompositions.
\newblock {\em Electron. J. Combin.}, 25(1.35):1--8, 2018.

\bibitem[Gen22]{genevois:median}
Anthony Genevois.
\newblock Median sets of isometries in {CAT}(0) cube complexes and some applications.
\newblock {\em Michigan Math. J.}, 71(3):487--532, 2022.

\bibitem[Gen23]{genevois:algebraic}
Anthony Genevois.
\newblock Algebraic properties of groups acting on median graphs.
\newblock {\em Preprint available at \mbox{sites.google.com/view/agenevois/books}}, 2023.

\bibitem[GL70]{gyarfaslehel:helly}
A.~Gy{\'a}rf\'as and J.~Lehel.
\newblock A {H}elly-type problem in trees.
\newblock In {\em Combinatorial theory and its applications, {I}-{III} ({P}roc. {C}olloq., {B}alatonf\"ured, 1969)}, volume~4 of {\em Colloq. Math. Soc. J\'anos Bolyai}, pages 571--584. North-Holland, Amsterdam-London, 1970.

\bibitem[Gol78]{golumbic1978trivially}
Martin~Charles Golumbic.
\newblock Trivially perfect graphs.
\newblock {\em Discrete Mathematics}, 24(1):105--107, 1978.

\bibitem[Gro87]{gromov:hyperbolic}
M.~Gromov.
\newblock Hyperbolic groups.
\newblock In {\em Essays in group theory}, volume~8 of {\em Math. Sci. Res. Inst. Publ.}, pages 75--263. Springer, New York, 1987.

\bibitem[Gy{\'a}87]{gyarfas:problems}
A.~Gy{\'a}rf\'as.
\newblock Problems from the world surrounding perfect graphs.
\newblock In {\em Proceedings of the {I}nternational {C}onference on {C}ombinatorial {A}nalysis and its {A}pplications ({P}okrzywna, 1985)}, volume 19.3-4, pages 413--441, 1987.

\bibitem[Hag08]{haglund:aspects}
Fr{\'e}d{\'e}ric Haglund.
\newblock Aspects combinatoires de la th{\'e}orie g{\'e}om{\'e}trique des groupes.
\newblock {\em Habilitation, Universit\'{e} Paris Sud}, 2008.

\bibitem[HJ63]{halesjewett:regularity}
A.~W. Hales and R.~I. Jewett.
\newblock Regularity and positional games.
\newblock {\em Trans. Amer. Math. Soc.}, 106:222--229, 1963.

\bibitem[Hol07]{holloway:embeddings}
Gemma~Lauren Holloway.
\newblock {\em Embeddings of {CAT}(0) cube complexes in products of trees}.
\newblock PhD thesis, University of Southampton. https://eprints.soton.ac.uk/66300, 2007.

\bibitem[Lea13]{leary:metric}
Ian~J. Leary.
\newblock A metric {K}an-{T}hurston theorem.
\newblock {\em J. Topol.}, 6(1):251--284, 2013.

\bibitem[Mie14]{miesch:injective}
Benjamin Miesch.
\newblock Injective metrics on cube complexes.
\newblock {\em arXiv:1411.7234}, 2014.

\bibitem[MW86]{monma1986intersection}
Clyde~L Monma and Victor~K Wei.
\newblock Intersection graphs of paths in a tree.
\newblock {\em Journal of Combinatorial Theory, Series B}, 41(2):141--181, 1986.

\bibitem[PKK{\etalchar{+}}13]{pawlikkozikkrawczyklasonmicektrotterwalczak:triangle:geometric}
Arkadiusz Pawlik, Jakub Kozik, Tomasz Krawczyk, Micha\l{} Laso\'n, Piotr Micek, William~T. Trotter, and Bartosz Walczak.
\newblock Triangle-free geometric intersection graphs with large chromatic number.
\newblock {\em Discrete Comput. Geom.}, 50(3):714--726, 2013.

\bibitem[PT23]{pournajafitrotignon:burling:1}
Pegah Pournajafi and Nicolas Trotignon.
\newblock Burling graphs revisited, part {I}: {N}ew characterizations.
\newblock {\em European J. Combin.}, 110(103686):1--24, 2023.

\bibitem[PV88]{promelvoigt:sparse:grahamrothschild}
Hans~J\"urgen Pr\"omel and Bernd Voigt.
\newblock A sparse {G}raham-{R}othschild theorem.
\newblock {\em Trans. Amer. Math. Soc.}, 309(1):113--137, 1988.

\bibitem[PV90]{promelvoigt:sparse:galaiwitt}
H.~J. Pr\"omel and B.~Voigt.
\newblock A sparse {G}allai-{W}itt theorem.
\newblock In {\em Topics in combinatorics and graph theory ({O}berwolfach, 1990)}, pages 747--755. Physica, Heidelberg, 1990.

\bibitem[Rad43]{rado:note}
R.~Rado.
\newblock Note on combinatorial analysis.
\newblock {\em Proc. London Math. Soc. (2)}, 48:122--160, 1943.

\bibitem[R{\"o}d90]{rodl:onramsey}
Vojt\v{e}ch R{\"o}dl.
\newblock On {R}amsey families of sets.
\newblock {\em Graphs Combin.}, 6(2):187--195, 1990.

\bibitem[Rol98]{roller:poc}
Martin Roller.
\newblock Poc sets, median algebras and group actions.
\newblock {\em Habilitationsschrift, Universit\"at Regensburg}, 1998.
\newblock Available on the arXiv at \mbox{arXiv:1607.07747}.

\bibitem[RT91]{rozoythiagarajan:event}
Brigitte Rozoy and P.~S. Thiagarajan.
\newblock Event structures and trace monoids.
\newblock {\em Theoret. Comput. Sci.}, 91(2):285--313, 1991.

\bibitem[RW26]{rzazewskiwalczak:polynomial}
Pawe\l{} Rzą\.zewski and Bartosz Walczak.
\newblock Polynomial-time recognition and maximum independent set in {B}urling graphs.
\newblock In {\em Graph-theoretic concepts in computer science}, volume 16124 of {\em Lecture Notes in Comput. Sci.}, pages 445--460. Springer, Cham, [2026] \copyright 2026.

\bibitem[San10]{santocanale:nice}
Luigi Santocanale.
\newblock A nice labelling for tree-like event structures of degree 3.
\newblock {\em Inform. and Comput.}, 208(6):652--665, 2010.

\bibitem[Wis21]{wise:structure}
Daniel~T. Wise.
\newblock {\em The structure of groups with a quasiconvex hierarchy}, volume 209 of {\em Annals of Mathematics Studies}.
\newblock Princeton University Press, Princeton, NJ, 2021.

\end{thebibliography}

\end{document}